\documentclass[reqno]{amsart}

\usepackage{amssymb,amsmath,amsthm,xcolor,enumerate,comment,hyperref,cleveref}
\usepackage[normalem]{ulem}
\usepackage[utf8]{inputenc}

\allowdisplaybreaks

\newtheorem{thrm}{Theorem}[section]
\newtheorem{cor}[thrm]{Corollary}
\newtheorem{lem}[thrm]{Lemma}
\newtheorem{prop}[thrm]{Proposition}

\theoremstyle{definition}
\newtheorem{defn}[thrm]{Definition}

\newtheorem{rem}[thrm]{Remark}

\crefrangeformat{equation}{#3(#1)#4--#5(#2)#6}

\crefname{thrm}{Theorem}{Theorems}
\crefname{lem}{Lemma}{Lemmas}
\crefname{cor}{Corollary}{Corollaries}
\crefname{prop}{Proposition}{Propositions}
\crefname{defn}{Definition}{Definitions}
\crefname{exm}{Example}{Examples}
\crefname{rem}{Remark}{Remarks}
\crefname{section}{Section}{Sections}
\crefname{equation}{\unskip}{\unskip}
\crefname{enumi}{\unskip}{\unskip}
\crefname{subsection}{Subsection}{Subsections}

\makeatletter
\newcommand{\mylabel}[2]{#2\def\@currentlabel{#2}\label{#1}}
\makeatother

\renewcommand{\iff}{\Leftrightarrow}

\newcommand{\impl}{\Rightarrow}

\newcommand{\id}{\mathrm{id}}

\newcommand{\cS}[1]{\mathcal S{(#1)}}
\newcommand{\cD}{\mathcal D}
\newcommand{\U}[1]{\mathcal U{(#1)}}
\newcommand{\cF}{\mathcal F}

\newcommand{\bbN}{\mathbb N}

\newcommand{\End}{\mathrm{End}}

\newcommand{\Hom}{\mathrm{Hom}}
\newcommand{\Der}{\mathrm{Der}}
\newcommand{\PDer}{\mathrm{PDer}}

\newcommand{\D}{\mathrm{D}}

\newcommand{\PD}{\mathrm{PD}}

\DeclareMathOperator{\im}{im}

\newcommand{\pr}{\mathrm{pr}}

\newcommand{\m}{{}^{-1}}

\newcommand{\0}{\theta}
\newcommand{\ve}{\varepsilon}
\newcommand{\e}{\epsilon}

\newcommand{\af}{\alpha}
\newcommand{\bt}{\beta}
\newcommand{\lb}{\lambda}
\newcommand{\Lb}{\Lambda}

\newcommand{\f}{\varphi}
\newcommand{\s}{\sigma}
\newcommand{\vt}{\vartheta}
\newcommand{\dl}{\delta}

\newcommand{\Sg}{\Sigma}

\newcommand{\A}{\mathcal A}

\newcommand{\B}{\mathcal B}
\newcommand{\M}{\mathcal M}

\newcommand{\tl}{\tilde}
\newcommand{\wtl}{\widetilde}

\newcommand{\sst}{\subseteq}
\newcommand{\ol}{\overline}

\newcommand{\kpar}[1]{K_{\mathrm{par}}(#1)}

\newcommand{\la}{\cdot}

\newcommand{\emp}{(\phantom{x})}

\newcommand{\cb}{\partial}
\newcommand{\arr}[1]{\overset{#1}{\to}}

\begin{document}

\title[Globalization of  cohomology in the sense of Alvares-Alves-Redondo]{Globalization of group cohomology in the sense of Alvares-Alves-Redondo}
	
	%
	\author{Mikhailo Dokuchaev}
	\address{Insituto de Matem\'atica e Estat\'istica, Universidade de S\~ao Paulo,  Rua do Mat\~ao, 1010, S\~ao Paulo, SP,  CEP: 05508--090, Brazil}
	\email{dokucha@gmail.com}
	\author{Mykola Khrypchenko}
	\address{Departamento de Matem\'atica, Universidade Federal de Santa Catarina, Campus Reitor Jo\~ao David Ferreira Lima, Florian\'opolis, SC,  CEP: 88040--900, Brazil \and Centro de Matem\'atica e Aplica\c{c}\~oes, Faculdade de Ci\^{e}ncias e Tecnologia, Universidade Nova de Lisboa, 2829--516 Caparica, Portugal}
	\email{nskhripchenko@gmail.com}
	\author{Juan Jacobo Sim\'on}
	\address{Departamento de Matem\'{a}ticas, Universidad de Murcia, 30071 Murcia, Spain}
	\email{jsimon@um.es} 
	\begin{abstract}
	 Recently E. R. Alvares, M. M. Alves and M. J. Redondo introduced  a cohomology   for a group $G$ with values in a module over the partial group algebra    $\kpar G,$  which is different from the partial group cohomology defined earlier by the  first two named authors of the present paper. Given a
unital partial action $\alpha $ of  $G$ on a (unital)  algebra $\A$ we con\-si\-der $\A$ as a  $\kpar G$-module in a natural way and study the globalization problem for the cohomology in the sense of  Alvares-Alves-Redondo with values in $\A.$ The problem is reduced to an extendibility property of cocycles. Furthermore, assuming that $\A$ is a product of blocks, we prove that any cocycle is globalizable, and globalizations of cohomologous cocycles are also cohomologous. As a consequence we obtain that the   Alvares-Alves-Redondo cohomology group  $H_{par}^n(G,\A)$ is isomorphic to the usual cohomology group $H^n(G,\M(\B)),$ where $\M(\B)$  is the multiplier algebra of  $\B$ and $\B$ is the algebra under the enveloping action of $\alpha .$  
	\end{abstract}

	\maketitle
	
	\section*{Introduction} In algebra  $2$-cocycles of groups  appeared as factor sets in the study of group extensions,  in  the  theory of projective group representations,  and in the construction of crossed products. Also $1$-cocycles were already hidden in Kummers' result,  widely known as Hilbert's Theorem 90.   Exel's idea of $C^*$-crossed products by twisted partial group actions \cite{E0} stimulated the purely algebraic treatment of the notion   in \cite{DES1}, the foundation of the theory  of partial projective group representations in  \cite{DN}, \cite{DN2},  \cite{DoNoPi} and \cite{DoSa}, 
as well as the definition and study  of group cohomology based on partial actions in \cite{DK} and  related extensions in  \cite{DK2} and  \cite{DK3}.\\ 

 A $G$-module over a group $G$  is an action of $G$ on an abelian group, and the starting point in \cite{DK} is the replacement of  a $G$-module by a partial $G$-module. The latter means  a unital partial action of $G$ on a commutative monoid. This fits nicely the concept   of a twisted partial group action in Exel's crossed products and the notion of the equivalence of twisted partial actions from \cite{DES2}.  The partial  group cohomology of   \cite{DK} found  applications to partial projective representations  \cite{DK}, \cite{DoSa},  to the construction of a Chase-Harrison-Rosenberg type seven terms exact sequence 
\cite{DoPaPi}, \cite{DoPaPiRo}, associated to a partial Galois extension of commutative rings \cite{DFP}, and  to the study of ideals of (global)  reduced $C^*$-crossed products  in \cite{KennedySchafhauser}. It also inspired the treatment of partial cohomology from the point of view of  Hopf algebras  \cite{BaMoTe}.\\

Observe  that the term ``partial $G$-module'', where $G$ is a group,  is being understood in two related but different senses.  Initially it was used under the name of  a ``partial $G$-space'' in \cite{DZh} when dealing with partial group representations.   The latter are closely related to partial actions (see \cite{D3}), and they  are governed by the partial group algebra $\kpar G$, whose purely algebraic version was introduced in \cite{DEP}.
So a  partial $G$-space  in \cite{DZh}  means a $\kpar G$-module, or equivalently,  a $K $-module $M$ equipped with a partial representation $G \to {\rm End}_{K } (M).$ This  is a natural point of view and it was adopted  in \cite{AAR} giving an  alternative approach to partial cohomology, with an appropriate notion of a trivial partial $G$-module. Despite the two ``partial cohomology'' theories can be defined using rather similar formulas, there is a significant difference between them due to the fact that the category of   $\kpar G$-modules is abelian, whereas the category of  partial $G$-modules from \cite{DK} is not even additive. A natural situation in which the two cohomologies can be compared is the case of  a $\kpar G$-module structure which comes from a unital partial action of $G$ on a commutative ring, and  a relation between the two approaches occurs only for $0$-cohomology.\\

The present paper deals mainly with the globalization problem of the cohomology  from  \cite{AAR}.  The technique worked out in  \cite{DES2} to deal with the globalization problem of twisted partial group actions on rings was  further developed in \cite{DKS} to obtain a globalization result for the partial group cohomology in the sense of \cite{DK}, where the partial $G$-module is not only a commutative semigroup but also a commutative unital ring, which is a direct product of indecomposable rings (blocks). It turns out, that our globalization technique can also be applied to the cohomology from \cite{AAR}, assuming,  that the base  $\kpar G$-module comes from a   partial action of $G$ on a product of (not necessarily commutative) blocks.\\

We begin by recalling some background in \cref{sec:Prelim}, and in \cref{sec:zeropar} we point out 
in \cref{H^0-and-H^0_par} the only direct relation we found between the cohomologies in  \cite{DK} and  \cite{AAR}. In the subsequent sections we deal only with cohomology in the sense of    \cite{AAR}. 
Given a $\kpar G$-module $M$,  the cohomology groups from \cite{AAR} with values in $M$ are denoted by    $H^n_{par}(G,M).$ Partial derivations were introduced 
in \cite{AAR}  as $K$-linear maps $\dl:\kpar G\to M$ satisfying a certain   Leibniz rule, and 
the cohomology group  $H^1_{par}(G,M)$ was characterized   as the quotient group of the partial derivations by the subgroup of the principal  partial derivations.  In  \cref{sec:1-partial} we correct an overview in the definition of a partial derivation given in \cite{AAR} and   offer one more interpretation of the elements of $H^1_{par}(G,M)$ which involves  some maps $f:G\to M$ satisfying a kind of $1$-cocycle identity (see \cref{H^1_par-in-terms-of-d:G->M}). The commutative subalgebra $B$ of  $\kpar G,$ defined in \cite{DE} in order to endow  $\kpar G$ with the structure of a partial crossed product, can be naturally  seen as a $\kpar G$-module, which plays the role of trivial module in the definition of the cohomology in   \cite{AAR}. In 
\cref{sec:projres} we  construct  a projective resolution   of the $\kpar G$-module $B$ adapting  some ideas of the free resolution from~\cite{DK}. \cref{H^n_par(G_M)=ker-dl^n-mod-im-dl^(n-1)} asserts that our projective resolution gives   the cohomology from     \cite{AAR}. Independently, Dessislava Kochloukova also produced  a projective resolution of $B$ exploring the crossed product structure of  $\kpar G.$  Dessislava's notes were the starting point for a collaboration which resulted in  the preprint \cite{AlDoKo}.\\ 

In \cref{sec:tildew} we begin our work with the globalization problem. Given a unital partial action of $G$ on a (not necessarily commutative) algebra $\A$ over a (commutative) ring $K$, we obtain  in a standard way a    $\kpar G$-module structure on $\A$ and study the globalization problem for the cohomology with values in such a module.    As a first step, we prove in \cref{w-glob-iff-exists-tilde-w} that a cocycle $w\in  Z^n_{par}(G, \A)$ is globalizable if and only if there exists a certain extension $\wtl w :G^n\to\A$  of $w$ which satisfies a ``more global'' $n$-cocycle equality.  Our global (usual) cocycles take values in the additive group of the multiplier algebra $\M (\B)$ of $\B$ where $\B$ is the algebra under the global action $\beta $ of $G,$ which is an enveloping action of $\alpha $ (i.e. a globalization of $\alpha $ with a certain minimality condition).   In order to construct $\wtl w,$ we assume, as in \cite{DES1} and \cite{DKS}, that $\A$ is a product of blocks. Our main technical work is done in \cref{sec:w'} in which for an arbitrary cocycle $w\in  Z^n_{par}(G, \A)$ we construct an $n$-cocycle $w' \in  Z^n_{par}(G, \A)$ cohomologous to 
$w$ and  suitable for a desired extension $\wtl{w'}$ (see 
\cref{w'-cohom-w}).  In \cref{sec:existsunique} we use  $\wtl{w'}$ to produce an extension  $\wtl{w}$ of $w$  needed for the application of \cref{w-glob-iff-exists-tilde-w}. This leads to
\cref{glob-exists} which asserts that any cocycle from $ Z^n_{par}(G, \A)$   is globalizable. Our uniqueness result is  \cref{uniqueness-of-glob} which says that globalizations of cohomologous $n$-cocycles from $ Z^n_{par}(G, \A)$  are also cohomologous. The two latter facts  imply our final result \cref{H^n(G_A)-cong-H^n(G_B)} which states that  
 $H_{par}^n(G,\A)$ is isomorphic to the classical cohomology group $H^n(G,\M(\B))$.

	\section{Preliminaries}\label{sec:Prelim}
	
	In this section we recall some basic notions and facts used in the sequel.   Our algebras will be over a commutative unital ring $K.$ A {\it partial action} $\alpha $ of a group $G$ on a non-necessarily unital   algebra   (or a ring) $\mathcal A$ is a family of two-sided ideals $\mathcal D_g $ of $\mathcal A$ and  algebra  (or ring) isomorphisms $\{\alpha _g : {\mathcal D} _{g\m} \to {\mathcal D}_g\mid g\in G\}$  such that
		\begin{enumerate}
			\item $\alpha _1 = {\rm id}_{\mathcal A},$
			\item $ \exists \; \alpha_h(a),  \; \exists \; \alpha_g (\alpha_h(a)) \; \impl   \;   \exists  \; \alpha _{gh}(a) \;\; \text{and} \;\;    \alpha _g (\alpha_h(a)) = \alpha _{gh}(a).$
		\end{enumerate}
		Here $\exists\;\af_g(a)$ means that $a\in\cD_{g\m}$.

	Replacing above the word ``algebra'' by ``semigroup'' we obtain the concept of a partial action of $G$ on a 
	semigroup. A partial action $\alpha $ is called {\it unital} if each ${\mathcal D}_g$ is unital, i.e.  
	${\mathcal D}_g = 1_g{\mathcal A},$ where $1_g$ is a central idempotent of ${\mathcal A}.$ If we replace ``algebra'' by ``set'', ``ideal'' by ``subset'' and ``isomorphism'' by ``bijection'', we come to the notion of a partial action of $G$ on an abstract  set.\\ 
	
	Let $\alpha = \{\af _g: {\cD }_{g\m} \to {\cD} _g\mid g\in G \}$  and $\alpha ' = \{\af ' _g: {\cD } '_{g\m} \to {\cD } ' _g\mid g\in G \}$
	be partial actions of a group $G$ on algebras $\A $ and $ \A ',$ respectively.  Recall from~\cite{AbadieTwo,DN2} that a {\it morphism}  of partial actions $(\A,\af) \to (\A',\af ')$   is a homomorphism of algebras $\varphi: \A\to \A '$ such that $\varphi(\cD _x)\subseteq \cD '_x$ and $\varphi\circ\af _x=\af '_x\circ\varphi$ on $\cD _{x^{-1}}$. Thus, partial actions of a group $G$ on algebras (rings, semigroups or sets) form a category.\footnote{Notice that isomorphic partial actions were called equivalent in \cite{DE}.}\\

	The partial group cohomology theory, as developed in \cite{DK}, is based on the concept of a unital partial $G$-module, which generally means a  unital partial action of $G$ on a commutative monoid. On the other hand, the cohomology theory introduced more recently in  \cite{AAR}  deals with (usual) modules over $\kpar G.$ 
	Recall that the {\it partial group algebra} $\kpar G$ of a group $G$ over a commutative ring $K$ can be seen as the semigroup algebra $K\cS{G}$, where $\cS G$ is the  monoid defined by R. Exel in \cite{E1} for an arbitrary group $G$ by means of generators $\{ [g] \; |\; g\in G\}$ and relations:
	\begin{align*}
		[g^{-1}][g][h]&=[g^{-1}][gh],\\
		{}[g][h][h^{-1}]&=[gh][h^{-1}],\\
		{}[g][1_G]&=[g],
	\end{align*}
	where $g, h \in G$ (consequently,  $[1_G][g]=[g]$). The elements  $e_g = [g] [g\m]$ are commuting idempotents in $ \cS G,$ and one easily obtains the useful relations 
	$$ 
	[g] e_h = e_{gh} [g],\ e_h [g] = [g] e_{g\m h},
	$$ for all $g,h \in G.$  According to \cite[Proposition 2.5]{E1}, each element   $s\in \cS G$ can be written as
	\begin{align}\label{stand-decomp}
		s = e_{h_1}  \ldots e_{h_k} [g],
	\end{align}  for some $k \in \bbN, h_1, \ldots , h_k, g \in G.$ One can assume that
	\begin{enumerate}
		\item $ h_i \neq   h_j $ for $i\neq j ,$\label{h_i-ne-h_j}
		\item $h_i \neq g$ and $h_i \neq 1_G$ for all $i.$\label{h_i-ne-g}
	\end{enumerate} 
	Under \cref{h_i-ne-g,h_i-ne-h_j} the decomposition \cref{stand-decomp} is unique up to the order of the idempotents $e_{h_i}$, and this was used by R. Exel to conclude that $\cS G$ is an inverse semigroup \cite[Theorem 3.4]{E1} (see also \cite{KL}).\\
	
	The defining relations of  $ \cS G$ are designed to guarantee that the map 
	\begin{align*}
		G\ni g \mapsto [g] \in   \cS G
	\end{align*}
	is a partial representation. More generally, we recall that a map $\pi : G \to S,$ where $S$ is a monoid, is called a {\it partial representation} (or a partial homomorphism)  if 
	
	\begin{enumerate}
		\item  $\pi (1_G) =  1_S,$
		\item $\pi (g\m)   \pi (g) \pi (h) = \pi (g\m)   \pi (gh),$
		\item $\pi (g)   \pi (h) \pi (h\m ) = \pi (g h)   \pi (h\m ),$
	\end{enumerate}  
	for all $g, h \in G.$ There is an evident bijective correspondence between the partial homomorphisms from $G$ into a unital  algebra $\A $ and the homomorphisms $\kpar G\to\A.$  
	
	The following fact, which is well known to the experts and whose verification is a direct exercise, results in a situation in which both partial cohomology theories are applicable.
	
	\begin{lem}\label{from-pMod-to-KparG-mod}
		Let $(\A,\af)$ be a unital partial action of a group $G$ on an algebra $\A.$ Then the map $\pi^\af:G\to\End_K(\A)$, $g\mapsto\pi^\af$, given by
		\begin{align}\label{pi^theta_g(a)=0_g(1_(g-inv)a)}
			\pi^\af_g(a)=\af_g(1_{g\m}a),
		\end{align} 
		is a partial representation of $G$. 
	\end{lem} It follows that  \cref{pi^theta_g(a)=0_g(1_(g-inv)a)} induces a $\kpar G$-module structure on $\A$ by means of
	\begin{align}\label{[g]a=pi^0_g(a)}
		[g]\cdot a=\pi^\af_g(a).
	\end{align}

	\section{The Alvares-Alves-Redondo cohomology}

	\subsection{The $0$-th cohomology group}\label{sec:zeropar}

	We begin with  the situation of \cref{from-pMod-to-KparG-mod} with commutative $\A $, so that we shall consider partial $G$-modules in a more restricted sense than  in \cite{DK}.  More precisely,  by a {\it partial $G$-module} we  shall mean  a commutative partial  $G$-module algebra\footnote{This is Hopf theoretic terminology (see \cite{CaenJan}).}, i.e.  a
	pair $(\A,\af)$, where $\A$ is a commutative algebra over some fixed field $K$ and $\af$ is a  partial action of $G$ on $\A .$ Moreover, we shall assume that $\af $ is unital, which is also expressed  by saying that the partial $G$-module is {\it unital}.

	Recall from~\cite{DK} that, given a unital partial $G$-module $(\A,\af)$, the $0$-th partial cohomology group $H^0(G,\A)$ of $G$ with values in $(\A,\af)$ is
	\begin{align}\label{H^0-defn}
		H^0(G,\A)=\{a\in\U{\A}\mid \forall g\in G:\ \af_g(1_{g\m}a)=1_ga\}.
	\end{align}
	For a left $\kpar G$-module $M$ the $0$-th cohomology group $H^0_{par}(G,M)$ of $G$ with values in $M$ was defined in~\cite{AAR} as
	\begin{align}\label{H^0_par-defn}
		H^0_{par}(G,M)=\{m\in M\mid\forall g\in G:\  [g]\la m=e_g\cdot m\},
	\end{align}
	reminding that $e_g = [g][g\m] \in \kpar G$.
	
	\begin{prop}\label{H^0-and-H^0_par}
		Let $(\A,\af)$ be a unital partial $G$-module. Consider the induced $\kpar G$-module structure on the underlying $K$-vector space of $\A$. Then $H^0_{par}(G,\A)$ is a unital $K$-subalgebra of $\A$ and
		\begin{align}\label{H^0(G_A)=U(H^0_par(G_A))}
			H^0(G,\A)=\U{H^0_{par}(G,\A)}.
		\end{align}
	\end{prop}
	\begin{proof}
		Indeed, by \cref{H^0_par-defn,pi^theta_g(a)=0_g(1_(g-inv)a),[g]a=pi^0_g(a)} an element $a\in\A$ belongs to $H^0_{par}(G,\A)$ if and only if for every $g\in G$
		\begin{align*}
			\af_g(1_{g\m}a)=\af_g(1_{g\m}\af_{g\m}(1_g a)),
		\end{align*}
		the latter being $\af_g(\af_{g\m}(1_g a))=1_g a$. Comparing this with \cref{H^0-defn}, we obtain the desired equality \cref{H^0(G_A)=U(H^0_par(G_A))}.
	\end{proof}

	\subsection{The $1$-st cohomology group}\label{sec:1-partial}

	Observe that the commutative subalgebra $B$ of $\kpar G$ considered in~\cite{AAR} (see also \cite{DE}) is exactly $KE(\cS G)$, where $E(\cS G)$ is the semilattice of idempotents of $\cS G$. The action of $\cS G$ on $E(\cS G)$ by conjugation
	\begin{align}\label{s-la-e=ses^(-1)}
		s\la e=ses\m
	\end{align}
	extends by linearity to a homomorphism $\kpar G\to \End_K(B)$, yielding thus a $\kpar G$-module structure on $B$. Let $\e:\kpar G\to B$ be the $K$-linear map defined on $s\in\cS G$ by
	\begin{align}\label{e(s)=ss^(-1)}
		\e(s)=ss\m.
	\end{align}  
	Observe from \cref{s-la-e=ses^(-1)} that $\e$ is a morphism of $\kpar G$-modules, and it coincides with the one defined in~\cite{AAR} before Lemma~3.2. 
	
	We now recall an interpretation of $H^1_{par}(G,M)$ obtained in~\cite{AAR}. Given a $\kpar G$-module $M$, a $K$-linear map $\dl:\kpar G\to M$ is called a {\it partial derivation}, if for all $s,t\in\cS G$
	\begin{align}\label{dl(st)=sdl(t)+e(t)dl(s)}
		\dl(st)=s\la\dl(t)+\e(st)\la\dl(s).
	\end{align}
	Notice that \cref{dl(st)=sdl(t)+e(t)dl(s)} is a corrected version of the definition given in~\cite{AAR}.

	The partial derivations of $\kpar G$ with values in $M$ form a $K$-space, which will be denoted by $\Der_{par}(G,M)$. A partial derivation $\dl\in\Der_{par}(G,M)$ is called {\it principal} (or inner), if there exists $m\in M$, such that for all $g\in G$
	\begin{align*}
		\dl([g])=[g]\la m-e_g\la m.
	\end{align*}
	The $K$-subspace of the principal partial derivations will be denoted by $\PDer_{par}(G,M)$. Theorem~3.4 from~\cite{AAR} states that there is an isomorphism of additive groups
	\begin{align}\label{H^1-cong-Der-mod-Int}
		H^1_{par}(G,M)\cong\Der_{par}(G,M)/\PDer_{par}(G,M).
	\end{align}
	
	We shall use the isomorphism \cref{H^1-cong-Der-mod-Int} to give another interpretation of the elements of $H^1_{par}(G,M)$ in terms of certain maps $f:G\to M$ satisfying a kind of $1$-cocycle identity.
	
	\begin{lem}\label{dl-on-idempotents}
		Let $\dl\in\Der_{par}(G,M)$. Then for any $e\in E(\cS G)$ we have 
		\begin{align}\label{dl(e)-is-zero}
			\dl(e)=0.
		\end{align}
	\end{lem}
	\begin{proof}
		Indeed, as $e^2=e$, we obtain by \cref{e(s)=ss^(-1),dl(st)=sdl(t)+e(t)dl(s)}
		\begin{align}\label{dl(e)=2edl(e)}
			\dl(e)=e\la\dl(e)+e\la\dl(e)=2e\la\dl(e).
		\end{align}
		Multiplying the both sides of \cref{dl(e)=2edl(e)} by $e$, we obtain 
			$e \cdot \dl(e)=2e\la\dl(e),$ which  gives $e\la\dl(e)=0.$ Then   \cref{dl(e)=2edl(e)}  implies  \cref{dl(e)-is-zero}. 
	\end{proof}
	
	\begin{lem}\label{dl(es)-formulas}
		For arbitrary $\dl\in\Der_{par}(G,M)$ and $s\in\cS G$, $e\in E(\cS G)$ one has
			\begin{enumerate}
				\item $\dl(es)=e\la\dl(s)$;\label{dl(es)=edl(s)}
				\item $\dl(se)=ses\m\la\dl(s)$.\label{dl(se)=ses^(-1)dl(s)}
			\end{enumerate}
	\end{lem}
	\begin{proof}
		Indeed, applying \cref{dl(st)=sdl(t)+e(t)dl(s)} with $s=e$ and $t=s$ and then using \cref{dl(e)-is-zero} we obtain \cref{dl(es)=edl(s)}. Similarly the application of \cref{dl(st)=sdl(t)+e(t)dl(s)} with $t=e$ together with \cref{dl(e)-is-zero} gives \cref{dl(se)=ses^(-1)dl(s)}.
	\end{proof}
	
	\begin{lem}\label{from-f-to-dl}
		Let $M$ be a $\kpar G$-module and $d:G\to M$ a map such that for all $g,h\in G$ 
		\begin{align}\label{e_gd(gh)=[g]d(h)+e_hd(g)}
			e_g\la d(gh)=[g]\la d(h)+e_{gh}\la d(g).
		\end{align}
		Then the $K$-linear map $\dl:\kpar G\to M$ defined by
		\begin{align}\label{dl(e[g])=edl([g])}
			\dl(e[g])=e\la d(g),
		\end{align}
		where $e\in E(\cS G)$, is a partial derivation.
	\end{lem}
	\begin{proof}
		First of all, we show that $\dl$ is well defined. Applying \cref{e_gd(gh)=[g]d(h)+e_hd(g)} with $g=h=1_G$, we obtain $d(1_G)=0$. Then substituting $h=g\m$ into \cref{e_gd(gh)=[g]d(h)+e_hd(g)}, we have $d(g)=-[g]\la d(g\m)$, whence 
		\begin{align}\label{d(g)-is-e_g.d(g)}
			d(g)=e_g\la d(g).
		\end{align} 
		Now in view of  \cref{stand-decomp} notice that $e[g]=f[h]$ in $\cS G$ if and only if $g=h$ and $e_ge=e_gf$. In this case
		\begin{align*}
			e\la d(g)=ee_g\la d(g)=fe_g\la d(g)=f\la d(g).
		\end{align*}
		
		Consider two arbitrary elements $e[g]$ and $f[h]$ of $\cS G$. Then their product is $e[g]f[g\m][gh]$, so by \cref{dl(e[g])=edl([g])}
		\begin{align}\label{dl(e[g]f[h])=e[g]f[g^(-1)]d(gh)}
			\dl(e[g]f[h])=e[g]f[g\m]\la d(gh).
		\end{align}
		Now we calculate using \cref{e_gd(gh)=[g]d(h)+e_hd(g)}
			\begin{align*}
				e[g]\la \dl(f[h])+\e(e[g]f[h])\la \dl(e[g])&=e[g]f\la d(h)+\e(e[g]f[h])e\la d(g)\\
				&=e[g]f[g\m][g]\la d(h)+e[g]f[g\m]e_{gh}e\la d(g)\\
				&=e[g]f[g\m]\la([g]\la d(h)+e_{gh}\la d(g))\\
				&=e[g]f[g\m]\la(e_g\la d(gh))\\
				&=e[g]f[g\m]\la d(gh),
			\end{align*}
		which in view of \cref{dl(e[g]f[h])=e[g]f[g^(-1)]d(gh)} shows that $\dl$ is a partial derivation.
	\end{proof}
	
	Let us denote by $\D(G,M)$ the $K$-vector space of the maps $d:G\to M$ which satisfy \cref{e_gd(gh)=[g]d(h)+e_hd(g)}.

	\begin{prop}\label{descr-of-Der}
		There is a bijective correspondence between the partial derivations of $\kpar G$ with values in $M$ and the elements of $\D(G,M)$.
	\end{prop}
	\begin{proof}
		Let $\dl\in\Der_{par}(G,M)$. By \cref{dl(st)=sdl(t)+e(t)dl(s),e(s)=ss^(-1)} and \cref{dl(es)-formulas}.\cref{dl(es)=edl(s)} we obtain
		\begin{align*}
			e_g\la\dl([gh])=\dl(e_g[gh])=\dl([g][h])=[g]\la\dl([h])+e_{gh}\la\dl([g]).
		\end{align*}
		So if we define $d:G\to M$ by 
		\begin{align}\label{d(g)=dl([g])}
			d(g)=\dl([g]),
		\end{align}
		then $d$ will satisfy \cref{e_gd(gh)=[g]d(h)+e_hd(g)}. Conversely, if $d\in\D(G,M)$, then $\dl:\kpar G\to M$ given by \cref{dl(e[g])=edl([g])} is a  partial derivation as was proved in \cref{from-f-to-dl}. 
		
		We now prove that the correspondence between $\dl$ and $d$ given by \cref{d(g)=dl([g]),dl(e[g])=edl([g])} is indeed bijective. If $d\mapsto\dl\mapsto d'$, then
		\begin{align*}
			d'(g)=\dl([g])=\dl(1_{\cS G}[g])=1_{\cS G}\la d(g)=d(g).
		\end{align*}
		And if $\dl\mapsto d\mapsto\dl'$, then
		\begin{align*}
			\dl'(e[g])=e\la d(g)=e\la\dl([g])=\dl(e[g]).
		\end{align*}
	\end{proof}
	
	Let us introduce one more notation:
	\begin{align*}
		\PD(G,M)=\{d:G\to M\mid\exists m\in M\ \forall g\in G:\ d(g)=[g]\la m-e_g\la m \}.
	\end{align*}
	Clearly, $\PD(G,M)$ is a $K$-subspace of $\D(G,M)$.
	
	\begin{thrm}\label{H^1_par-in-terms-of-d:G->M}
		Let $M$ be a $\kpar G$-module. Then $H^1_{par}(G,M)$ is isomorphic to the quotient of the additive group of $\D(G,M)$ modulo the subgroup $\PD(G,M)$.
	\end{thrm}
	\begin{proof}
		This follows from \cref{H^1-cong-Der-mod-Int,descr-of-Der} and the observation that the principal partial derivations of $\kpar G$ with values in $M$ correspond to the elements of $\PD(G,M)$.
	\end{proof}
	
	\begin{cor}\label{H^1-as-the-group-of-functions}
		Let $(\A,\af)$ be a unital partial $G$-module. As in \cref{H^0-and-H^0_par} we consider the corresponding $\kpar G$-module structure on $\A$. Then $H^1_{par}(G,\A)$ is isomorphic to the quotient of the additive group of functions\footnote{Observe that such functions automatically satisfy $f(g)\in 1_g\A$ in view of \cref{d(g)-is-e_g.d(g)}.}
			\begin{align*}
			\{f:G\to\A\mid \forall g\in G:\ 1_gf(gh)=\af_g(1_{g\m} f(h))+1_{gh}f(g)\}
			\end{align*}
		by the subgroup 
		\begin{align*}
		\{f:G\to\A\mid \exists a\in \A\ \forall g\in G:\ f(g)=\af_g(1_{g\m}a)-1_ga\}.
		\end{align*}
	\end{cor}
	\cref{H^1-as-the-group-of-functions} permits us to compare $H^1_{par}(G,\A)$ with the $1$-st partial cohomology group $H^1(G,\A)$. Recall from~\cite{DK} that, in the setting of \cref{H^1-as-the-group-of-functions}, $H^1(G,\A)$ is the quotient of the \textit{multiplicative} group of functions
	\begin{align*}
	\{f:G\to\A\mid \forall g\in G:\ f(g)\in\U{1_g\A}\mbox{ and }1_gf(gh)=\af_g(1_{g\m} f(h))f(g)\}
	\end{align*}
	by the subgroup
	\begin{align*}
	\{f:G\to\A\mid \exists a\in \U{\A}\ \forall g\in G:\ f(g)=\af_g(1_{g\m}a)a\m\}.
	\end{align*}
	Thus, $H^1_{par}(G,\A)$ is an ``additive'' analogue of $H^1(G,\A)$.

	\subsection{A projective resolution of the $\kpar G$-module $B$}\label{sec:projres}
	
	We are going to characterize the elements of $H^n_{par}(G,M)$ as classes of functions $f:G\to M$ satisfying certain $n$-cocycle identity, as we did in \cref{sec:1-partial} for $n=1$. To this end, we adapt some ideas  from~\cite{DK} to the case of  $\kpar G$-modules.
	
	\begin{lem}\label{dir-sum-of-Re_i}
		Let $R$ be a unital ring and $\{e_i\}_{i\in I}\sst E(R)$ a set of idempotents of $R$. Then the left $R$-module $\bigoplus_{i\in I}Re_i$ is projective.
	\end{lem}
	\begin{proof}
		Indeed, each $Re_i$ is a projective left $R$-module, since $Re_i\oplus R(1_R-e_i)$ is isomorphic to $R$, a free $R$-module of rank $1$. Now, a direct sum of projective modules is projective (see, e.g.,~\cite[Lemma~2.9~(iii)]{Passman}).
	\end{proof}
	
	Let $n\in\bbN$ and $g_1,\dots,g_n\in G$. As in~\cite{DK}, we shall use the following notation:
	\begin{align*}
		e_{(g_1,\dots,g_n)}=e_{g_1}e_{g_1g_2}\dots e_{g_1\dots g_n}\in E(\cS G).
	\end{align*}
	
	\begin{defn}\label{P_n-defn}
		Define
		\begin{align*}
			P_0&=\kpar G,\\
			P_n&=\bigoplus_{g_1,\dots,g_n\in G}\kpar G e_{(g_1,\dots, g_n)},\ n\in\bbN.
		\end{align*}
	\end{defn}
	By \cref{dir-sum-of-Re_i} each $P_n$ is a projective $\kpar G$-module. It would be convenient to us to have an equivalent description of the modules $P_n$ which reminds the free resolution $R_n$ from \cite{DK}.
	
	\begin{rem}\label{s(g_1...g_n)-basis}
		For each $n\in\bbN$ the module $P_n$ is isomorphic, as a $K$-vector space, to the vector space over $K$ with basis
		\begin{align}\label{s(g_1...g_n):s^(-1)s<=e_(g_1...g_n)}
			\{s(g_1,\dots,g_n)\mid s\in\cS G,\ g_1,\dots,g_n\in G,\ s\m s\le e_{(g_1,\dots,g_n)}\},
		\end{align}
		where
		\begin{align}\label{s(g_1...g_n)=t(h_1...h_n)<=>...}
			s(g_1,\dots,g_n)=t(h_1,\dots,h_n)&\iff 
			\begin{cases}
				(g_1,\dots,g_n)=(h_1,\dots,h_n),\\
				se_{(g_1,\dots,g_n)}=te_{(h_1,\dots,h_n)}.
			\end{cases}
		\end{align}
	\end{rem}
	\begin{proof}
		Indeed, the elements $se_{(g_1,\dots,g_n)}$, where $s\in \cS G$ and $g_1,\dots,g_n\in G$, form a basis of the $K$-vector space $P_n$. Clearly, such $se_{(g_1,\dots,g_n)}$ may be identified with $s(g_1,\dots,g_n)$, if \cref{s(g_1...g_n)=t(h_1...h_n)<=>...} is assumed.
		It remains to observe that
		\begin{align}\label{t=se_(g_1...g_n)<=>t^(-1)t<=e_(g_1...g_n)}
			\exists s\in\cS G:\ t=se_{(g_1,\dots,g_n)}\iff t=te_{(g_1,\dots,g_n)}\iff t\m t\le e_{(g_1,\dots,g_n)}.
		\end{align}
	\end{proof}
	We extend the characterization of $P_n$ from \cref{s(g_1...g_n)-basis} to $n=0$ by identifying $P_0$ with the $K$-vector space with basis
	\begin{align*}
		\{s\emp\mid s\in\cS G\}.
	\end{align*}
	
	\begin{defn}\label{d_n:P_n->P_(n+1)-defn}
		Define $K$-linear maps $\cb_0:P_0\to B$ and $\cb_n:P_n\to P_{n-1}$, $n\in\bbN$, as follows
		\begin{align}
			\cb_0(s\emp)&=ss\m,\label{d_0(s())=ss^(-1)}\\
			\cb_1(s(g))&=s([g]\emp-\emp),\label{cb_1(s(g)=s([g]()-()}\\
			\cb_n(s(g_1,\dots,g_n))&=s([g_1](g_2,\dots,g_n)\notag\\
			&\quad+\sum_{i=1}^{n-1}(-1)^i(g_1,\dots,g_ig_{i+1},\dots,g_n)\notag\\
			&\quad+(-1)^n(g_1,\dots,g_{n-1})).\label{cb_n(s(g_1...g_n)=...)}
		\end{align}	
	\end{defn}
	Observe that $\cb_n$, $n\in\bbN\cup\{0\}$, are morphisms of $\kpar G$-modules. Indeed, this is trivial for $n\in\bbN$, and for $n=0$ one should remember that the $\kpar G$-module structure on $B$ comes from the action of $\cS G$ on $E(\cS G)$ by conjugation.
	
	Our aim is to prove that \cref{d_n:P_n->P_(n+1)-defn} gives a projective resolution of $B$ in the category of $\kpar G$-modules. To this end, we shall show that the sequence $\{P_n\}_{n\ge-1}$, where $P_{-1}$ denotes $B$, admits a contracting homotopy similar to $\{\s_n\}_{n\ge -1}$ from~\cite[Definition~4.7]{DK}. Let $\eta:\cS G\to G$ be the semigroup homomorphism which maps $e[g]\in\cS G$ to $g$. As in~\cite[Lemma~4.8~(ii)]{DK}, one can  easily prove that 
	\begin{align}\label{s=ss^(-1)[eta(s)]}
		s=ss\m[\eta(s)].
	\end{align}
	
	\begin{defn}\label{sigma_n-defn}
		Define $K$-linear maps $\s_n:P_n\to P_{n+1}$, $n\in\bbN\cup\{-1,0\}$, as follows
		\begin{align}
			\s_{-1}(e)&=e\emp,\label{sigma_(-1)(e)=e()}\\
			\s_0(s\emp)&=ss\m(\eta(s)),\notag\\
			\s_n(s(g_1,\dots,g_n))&=ss\m(\eta(s),g_1,\dots,g_n),\ n\in\bbN.\notag
		\end{align}
	\end{defn}
	Since $s\le [\eta(s)]$ by \cref{s=ss^(-1)[eta(s)]}, we have that 
	\begin{align}\label{ss^(-1)<=e_eta(s)}
		ss\m\le e_{\eta(s)}
	\end{align} 
	for all $s\in\cS G$, and if moreover $s\m s\le e_{(g_1,\dots,g_n)}$, then
	\begin{align*}
		ss\m=s(s\m s)s\m\le[\eta(s)]s\m s[\eta(s)]\m\le[\eta(s)]e_{(g_1,\dots,g_n)}[\eta(s)]\m=e_{(\eta(s),g_1,\dots,g_n)}.
	\end{align*}
	Thus, $\s_n$, $n\in\bbN\cup\{-1,0\}$, are well defined.
	
	\begin{lem}\label{sigma-is-contr-homotopy}
		We have that
		\begin{align}
			\cb_0\circ\s_{-1}&=\id_B,\label{d_0-sigma_(-1)=id}\\
			\cb_{n+1}\circ\s_n+\s_{n-1}\circ\cb_n&=\id_{P_n},\ n\in\bbN\cup\{0\}.\label{d_(n+1)-sigma_n+sigma_(n-1)-d_n=id}
		\end{align}
	\end{lem}
	\begin{proof}
		It suffices to verify \cref{d_0-sigma_(-1)=id,d_(n+1)-sigma_n+sigma_(n-1)-d_n=id} on the $K$-basis \cref{s(g_1...g_n):s^(-1)s<=e_(g_1...g_n)} of $P_n$, $n\in\bbN\cup\{-1,0\}$. Equality \cref{d_0-sigma_(-1)=id} is a straightforward consequence of \cref{sigma_(-1)(e)=e(),d_0(s())=ss^(-1)}. To prove \cref{d_(n+1)-sigma_n+sigma_(n-1)-d_n=id}, one can follow the proof of~\cite[Lemma~4.9]{DK}, remembering that $\Gamma(g)$ is now $[g]$ for all $g\in G$ and removing the unnecessary idempotents $\epsilon_{(g_1,\dots,g_n)}$, where $g_1,\dots,g_n\in G$.
	\end{proof}
	
	\begin{prop}\label{P_n-is-res-of-B}
		The sequence
		\begin{align}\label{...->P_1->P_0->B}
			\dots\arr{\cb_2}P_1\arr{\cb_1}P_0\arr{\cb_0} B\to 0
		\end{align}
		is a projective resolution of $B$ in the category of $\kpar G$-modules.
	\end{prop}
	\begin{proof}
		In view of \cref{dir-sum-of-Re_i} we only need to prove that \cref{...->P_1->P_0->B} is exact. Exactness in $B$ is just \cref{d_0-sigma_(-1)=id}. The inclusion $\ker{\cb_n}\sst\im{\cb_{n+1}}$, $n\in\bbN\cup\{0\}$, is a trivial consequence of \cref{d_(n+1)-sigma_n+sigma_(n-1)-d_n=id}. For the converse inclusion, one may prove by induction on $n$ that $\cb_n\circ\cb_{n+1}\circ\s_n=0$ (see, e.g.,~\cite[p.~115]{Maclane}). This will guarantee that $\cb_n\circ\cb_{n+1}=0$, if we show that $\s_n(P_n)$ generates $P_{n+1}$ as a $\kpar G$-module. Let $s(g_1,\dots,g_{n+1})\in P_{n+1}$. Consider $t=[\eta(s)\m]s[g_1]\in\cS G$. Since 
		\begin{align*}
			se_{g_1}=se_{(g_1,\dots,g_{n+1})}e_{g_1}=se_{(g_1,\dots,g_{n+1})}=s
		\end{align*}
		by \cref{t=se_(g_1...g_n)<=>t^(-1)t<=e_(g_1...g_n)}, we obtain using \cref{s=ss^(-1)[eta(s)]} that
		\begin{align}\label{tt^(-1)=s^(-1)s}
			tt\m=[\eta(s)\m]se_{g_1}s\m[\eta(s)]=[\eta(s)\m]ss\m[\eta(s)]=s\m s.
		\end{align}
		Furthermore, by \cref{ss^(-1)<=e_eta(s)}
		\begin{align*}
			e_{\eta(s)}s=e_{\eta(s)}ss\m s=ss\m s=s,
		\end{align*}
		whence
		\begin{align}\label{t^(-1)t<=e_(g_2...g_(n+1))}
			t\m t=[g\m_1]s\m e_{\eta(s)}s[g_1]=[g\m_1]s\m s[g_1]\le [g\m_1]e_{(g_1,\dots,g_{n+1})}[g_1]\le e_{(g_2,\dots, g_{n+1})}.
		\end{align}
		Finally, 
		\begin{align}\label{eta(t)=g_1}
			\eta(t)=\eta([\eta(s)\m]s[g_1])=\eta(s)\m\eta(s)g_1=g_1.
		\end{align}
		Equalities \cref{tt^(-1)=s^(-1)s,eta(t)=g_1} imply that
		\begin{align*}
			s(g_1,\dots,g_{n+1})=s\cdot s\m s(g_1,\dots,g_{n+1})=s\la\s_n(t(g_2,\dots,g_{n+1}))
		\end{align*}
		holds formally, and \cref{t^(-1)t<=e_(g_2...g_(n+1))} is now used to show that $t(g_2,\dots,g_{n+1})\in P_n$, so that $\s_n(t(g_2,\dots,g_{n+1}))$ indeed makes sense.  
	\end{proof}
	
	\begin{defn}\label{C^n_par(G_M)-defn}
		Let $M$ be a $\kpar G$-module. Define the following additive groups
		\begin{align}
			C_{par}^0(G,M)&=M,\notag\\
			C_{par}^n(G,M)&=\{f:G^n\to M\mid f(g_1,\dots,g_n)\in e_{(g_1,\dots,g_n)}\la M\},\ n\in\bbN.\label{C^0_par(G_M)=...}
		\end{align}
	\end{defn}
	
	\begin{lem}\label{Hom(P_n_M)-in-terms-of-C^n_par}
		Let $M$ be a $\kpar G$-module. Then
		\begin{align*}
			\Hom_{\kpar G}(P_n,M)\cong C_{par}^n(G,M).
		\end{align*}
	\end{lem}
	\begin{proof} 	
		The case $n=0$ is explained by the fact that $P_0$ is a free $\kpar G$-module of rank $1$. Now let $n\in\bbN$ and observe using \cref{P_n-defn} and the standard isomorphism $\Hom(\bigoplus A_i,-) \cong \prod\Hom(A_i,-)$ (see, for example, \cite[p. 25]{Maclane}) that
			\begin{align*}
				\Hom_{\kpar G} (P_n, M)&\cong \prod_{g_1,\dots,g_n\in G}\Hom_{\kpar G}(\kpar Ge_{(g_1,\dots,g_n)},M)\\
				&\cong\prod_{g_1,\dots,g_n\in G}e_{(g_1,\dots,g_n)}\cdot M  \\
				&\cong C_{par}^n(G,M).
			\end{align*}
		\end{proof}

		\begin{rem}\label{Hom(P_n_M)-cong-C_par^n(G_M)-precise-form}
			With respect to the isomorphism from \cref{Hom(P_n_M)-in-terms-of-C^n_par}  any $\varphi\in\Hom_{\kpar G}(P_n,M)$ is mapped  to $f_\varphi\in C_{par}^n(G,M)$, where
			\begin{align*}
				f_\varphi(g_1,\dots,g_n)=\varphi(e_{(g_1,\dots,g_n)}(g_1,\dots,g_n)).
			\end{align*}
			Conversely,  each $f'\in C_{par}^n(G,M)$  corresponds to  $\varphi\in\Hom_{\kpar G}(P_n,M)$ defined by 
			\begin{align}\label{vf_f(s(g_1...g_n))=s-la-f(g_1...g_n)}
				\varphi(s(g_1,\dots,g_n))=s\la f'(g_1,\dots,g_n).
			\end{align}
		\end{rem}

	\begin{defn}\label{dl^n-defn}
		Let $M$ be a $\kpar G$-module and $n\in\bbN\cup\{0\}$. Define the $K$-linear map $\dl^n:C^n_{par}(G,M)\to C^{n+1}_{par}(G,M)$ as follows:
		\begin{align}
			(\dl^0 m)(g)&=[g]\la m-e_g\la m,\ m\in C^0_{par}(G,M),\label{(dl^0m)(g)=...}\\
			(\dl^nf)(g_1,\dots,g_{n+1})&=[g_1]\la f(g_2,\dots,g_{n+1})\notag\\
			&\quad +\sum_{i=1}^n (-1)^i e_{g_1\dots g_i}\la f(g_1,\dots,g_ig_{i+1},\dots,g_{n+1})\notag\\
			&\quad + (-1)^{n+1} e_{g_1\dots g_{n+1}}\la f(g_1,\dots,g_n),\ n\in\bbN,\ f\in C^n_{par}(G,M).\label{(dl^nf)(g_1...g_n)=...}
		\end{align}
	\end{defn}
	
	\begin{lem}\label{dl:C^n_par(G_M)->C^(n+1)_par(G_M)}
		For all $n\in\bbN\cup\{0\}$ and $f\in C_{par}^n(G,M)$ we have 
		\begin{align}\label{dl^nf=f-circ-cb_(n+1)}
			\dl^nf=f\circ\cb_{n+1},
		\end{align}
		where $f$ and $\dl^nf$ are identified with the morphisms from $\Hom_{\kpar G}(P_n,M)$ as in \cref{Hom(P_n_M)-in-terms-of-C^n_par}. In particular,
		\begin{align*}
			C^0_{par}(G,M)\arr{\dl^0}C^1_{par}(G,M)\arr{\dl^1}\dots
		\end{align*} 
		is a cochain complex of abelian groups.
	\end{lem}
	\begin{proof}
		It suffices to verify \cref{dl^nf=f-circ-cb_(n+1)} on the generators 
		\begin{align}\label{e_(g_1...g_n)(g_1...g_n)-generators}
		\{e_{(g_1,\dots,g_{n+1})}(g_1,\dots,g_{n+1})\mid g_1,\dots,g_{n+1}\in G\}
		\end{align}
		of $P_{n+1}$. We first consider the case $n=0$. Let $m\in C^0_{par}(G,M)$. Then, as an element of $\Hom_{\kpar G}(P_0,M)$, $m$ sends $s\emp$ to $s\la m$. Using \cref{cb_1(s(g)=s([g]()-(),vf_f(s(g_1...g_n))=s-la-f(g_1...g_n),(dl^0m)(g)=...}, we have
		\begin{align*}
			m\circ \cb_1(e_g(g))&=m(e_g([g]\emp-\emp))=m([g]\emp-e_g\emp)\\
			&=[g]\la m-e_g\la m=(\dl^0m)(g)\\
			&=e_g\la (\dl^0m)(g)=(\dl^0m)(e_g(g)).
		\end{align*}
		Now let $n\in\bbN$ and $f$ be a function from $C_{par}^n(G,M)$. By \cref{cb_n(s(g_1...g_n)=...),vf_f(s(g_1...g_n))=s-la-f(g_1...g_n),(dl^nf)(g_1...g_n)=...,C^0_par(G_M)=...}
		\begin{align*}
			f\circ\cb_{n+1}(e_{(g_1,\dots,g_{n+1})}(g_1,\dots,g_{n+1}))&=e_{(g_1,\dots,g_{n+1})}\la([g_1]\la f(g_2,\dots,g_{n+1})\\
			&\quad +\sum_{i=1}^n (-1)^i f(g_1,\dots,g_ig_{i+1},\dots,g_{n+1})\\
			&\quad + (-1)^{n+1} f(g_1,\dots,g_n))\\
			&=e_{(g_1,\dots,g_{n+1})}\la([g_1]\la f(g_2,\dots,g_{n+1})\\
			&\quad +\sum_{i=1}^n (-1)^i e_{g_1\dots g_i}\la f(g_1,\dots,g_ig_{i+1},\dots,g_{n+1})\\
			&\quad + (-1)^{n+1} e_{g_1\dots g_{n+1}}\la f(g_1,\dots,g_n))\\
			&=(\dl^nf)(e_{(g_1,\dots,g_{n+1})}(g_1,\dots,g_{n+1})).
		\end{align*}
	\end{proof}
	
	\begin{defn}\label{Z^n_par-and-B^n_par}
		Denote $\ker{\dl^n}$ by $Z^n_{par}(G,M)$, $n\in\bbN\cup\{0\}$, and $\im{\dl^n}$ by $B^n_{par}(G,M)$, $n\in\bbN$, where $\dl^n$ is given by \cref{(dl^0m)(g)=...,(dl^nf)(g_1...g_n)=...}.
	\end{defn}
	
	\begin{thrm}\label{H^n_par(G_M)=ker-dl^n-mod-im-dl^(n-1)}
		Let $G$ be a group and $M$ a $\kpar G$-module. Then $H^0_{par}(G,M)\cong Z^0_{par}(G,M)$ and $H^n_{par}(G,M)\cong Z^n_{par}(G,M)/B^n_{par}(G,M)$.
	\end{thrm}
	\begin{proof}
		This follows from \cref{P_n-is-res-of-B,Hom(P_n_M)-in-terms-of-C^n_par,dl:C^n_par(G_M)->C^(n+1)_par(G_M)}.
	\end{proof}
	
	\begin{rem}
		For a $\kpar G$-module $\A$ coming from a partial $G$-module $(\A,\af)$ and $n\in\bbN$ we have 
		in view of \cref{from-pMod-to-KparG-mod}
		\begin{align*}
			C_{par}^n(G,\A)=\{f:G^n\to \A\mid f(g_1,\dots,g_n)\in 1_{(g_1,\dots,g_n)}\A\},
		\end{align*}
		where 
			\begin{align*}
				1_{(g_1,\dots,g_n)}=1_{g_1}1_{g_1g_2}\dots 1_{g_1\dots g_n}\in \A .
		\end{align*}
		Then formulas \cref{(dl^0m)(g)=...,(dl^nf)(g_1...g_n)=...} take the following form
		\begin{align}
			(\dl^0 a)(g)&=\af_g(1_{g\m}a)-1_ga,\ a\in C^0_{par}(G,\A)=\A,\label{(dl^0a)(g)=af_g(1_(g-inv)a)-1_ga}\\
			(\dl^nf)(g_1,\dots,g_{n+1})&=\af_{g_1}\left(1_{g\m_1}f(g_2,\dots,g_{n+1})\right)\notag\\
			&\quad +\sum_{i=1}^n (-1)^i 1_{g_1\dots g_i}f(g_1,\dots,g_ig_{i+1},\dots,g_{n+1})\notag\\
			&\quad + (-1)^{n+1} 1_{g_1\dots g_{n+1}} f(g_1,\dots,g_n),   \label{(dl^nf)(g_1...g_(n+1))}
		\end{align} 	 
		$n\in\bbN, f\in C^n_{par}(G,M)$. 
	\end{rem}

	\section{Globalization}

	\subsection{From globalization to an extendibility property}\label{sec:tildew}
	Throughout this section $\af$ will be a unital partial action of a group $G$ on a (unital) algebra $\A$. We regard $\A$ as a $\kpar G$-module in a natural way (see \cref{from-pMod-to-KparG-mod}). We also fix $(\B,\bt)$ an enveloping action of $\af$ (see \cite[Definition 4.2]{DE}) with an injective morphism $\f:(\A ,\af)\to(\B,\bt)$.  The algebra $\B$ does not always have  an identity element, and for our   technique we need  to have  a unital algebra.  Instead of assuming that $\B $ has $1_{\B} $, we shall work more generally with the multiplier algebra $\M(\B)$ of $\B$.
		
	We recall that  the {\it multiplier algebra} $\M (\B)$ of an algebra $\B$ is the
	set $$\M (\B)= \{(R,L) \in {\rm End}(_{\B} \B) \times {\rm End}(\B_{\B}) : (aR)b = a(Lb) \, \mbox {for all}\, a,b
	\in \B \}$$ with component-wise addition and multiplication (for more details see \cite{DdRS,DE}). Here we use the 
	right-hand side notation for  left $\B$-module homomorphisms,  i.e.  we write $b \mapsto
	b\gamma $  for  $\gamma : {_{\B} \B} \to {_{\B} \B},$ while for a right $\B$-module homomorphism
	$\gamma : \B _{\B} \to \B _{\B}$ the usual notation is used: $b \mapsto \gamma b.$ 
	
	For a multiplier  $\gamma = (R,L) \in \M (\B)$ and $b \in \B$ we  set $b \gamma = b R$ and $\gamma b = L b.$ Thus
	one always has $(a \gamma ) b = a (\gamma b)$ for arbitrary $a,b \in \B$.
		
	The action $\bt$ induces an action $\bt^*$ of $G$ on $\M(\B)$, where $\bt^*_g(u)=\bt_g u\bt\m_g$ for $u\in\M(\B)$ and $g\in G$. Denote by $C^n(G,\M(\B))$, $Z^n(G,\M(\B))$, $B^n(G,\M(\B))$ and 
	$H^n(G,\M(\B))$ the corresponding (abelian) groups of $n$-cochains, $n$-cocycles, $n$-coboundaries and $n$-cohomologies of $G$ with values in the additive group of $\M(\B)$. 
	\begin{defn}\label{restr-defn}
		Given $n\in\bbN$ and $u\in C^n(G,\M(\B))$, define the {\it restriction} of $u$ to $\A$ to be the map $w:G^n\to\A$, such that
		\begin{align}\label{w-is-restr-of-u}
			\f(w(g_1,\dots,g_n))=\f(1_{(g_1,\dots,g_n)})u(g_1,\dots,g_n),
		\end{align}
		where $g_1,\dots,g_n\in G$. If $n=0$ and $u\in C^0(G,\M(\B))=\M(\B)$, then $w$ is the element of $\A$, satisfying \cref{w-is-restr-of-u}, in which $1_{(g_1,\dots,g_n)}$ means $1_\A$. 
	\end{defn}
	Notice that in \cref{w-is-restr-of-u} we could replace $\f(1_{(g_1,\dots,g_n)})u(g_1,\dots,g_n)$ by its ``left version'' $u(g_1,\dots,g_n)\f(1_{(g_1,\dots,g_n)})$. But in fact the two versions coincide. Indeed, $\f(\A)$ is an ideal of $\B$, so $u$ restricted to $\f(\A)$ is a multiplier of $\f(\A)$. Since $\f(1_{(g_1,\dots,g_n)})$ is a central idempotent of $\f(\A)$, we have $\f(1_{(g_1,\dots,g_n)})u(g_1,\dots,g_n)=\f(u(g_1,\dots,g_n)1_{(g_1,\dots,g_n)})$ by \cite[Remark 5.2]{DK2}. The observation that the multipliers of $\B$ (and of $\f(\A)$ as well) ``commute'' with central idempotents of $\f(\A)$ will be implicitly used several times in what follows.

	We will write $\rho(u)=w$ when $w$ is a restriction of $u$. Clearly, $\rho(u)\in C^n_{par}(G,\A)$, as the right-hand side of \cref{w-is-restr-of-u} is stable under the multiplication by $\f(1_{(g_1,\dots,g_n)})$ on the left.
	
	\begin{prop}\label{restr-hom-of-cohom-groups}
		The restriction map $\rho:C^n(G,\M(\B))\to C^n_{par}(G,\A)$ induces a homomorphism of the cohomology groups $H^n(G,\M(\B))\to H^n_{par}(G,\A)$.
	\end{prop}
	\begin{proof}
		It is readily seen by \cref{w-is-restr-of-u} that $\rho$ is a homomorphism, so we only need to show that $\rho$ commutes with the coboundary operators. Let $n=0$ and $u\in C^0(G,\M(\B))=\M(\B)$. Then for all $g\in G$ by \cref{w-is-restr-of-u} and the fact that $\f$ is a morphism of partial actions we have
		\begin{align*}
			\f((\dl^0\rho(u))(g))&=\f(\af_g(1_{g\m}\rho(u))-1_g\rho(u))\\
			&=\bt_g(\f(1_{g\m})\f(\rho(u)))-\f(1_g)\f(\rho(u))\\
			&=\bt_g(\f(1_{g\m})\f(1_\A)u)-\f(1_g)\f(1_\A)u\\
			&=\f(1_g)(\bt^*_g(u)-u)\\
			&=\f(1_g)(\dl^0u)(g)\\
			&=\f(\rho(\dl^0u)(g)),
		\end{align*}
		whence $\dl^0\rho(u)=\rho(\dl^0u)$.
		
		Consider now $n\in\bbN$ and $u\in C^n(G,\M(\B))$. For arbitrary $g_1,\dots,g_{n+1}\in G$, using \cref{w-is-restr-of-u} as above, one has
		\begin{align*}
			\f((\dl^n\rho(u))(g_1,\dots,g_{n+1}))&=\f\Big(\af_{g_1}\Big(1_{g\m_1}\rho(u)(g_2,\dots,g_{n+1})\Big)\\
			&\quad +\sum_{i=1}^n (-1)^i 1_{g_1\dots g_i}\rho(u)(g_1,\dots,g_ig_{i+1},\dots,g_{n+1})\\
			&\quad + (-1)^{n+1} 1_{g_1\dots g_{n+1}}\rho(u)(g_1,\dots,g_n)\Big)\\
			&=\bt_{g_1}\Big(\f\Big(1_{g\m_1}\Big)\f(\rho(u)(g_2,\dots,g_{n+1}))\Big)\\
			&\quad +\sum_{i=1}^n (-1)^i \f(1_{g_1\dots g_i})\f(\rho(u)(g_1,\dots,g_ig_{i+1},\dots,g_{n+1}))\\
			&\quad + (-1)^{n+1} \f(1_{g_1\dots g_{n+1}})\f(\rho(u)(g_1,\dots,g_n))\\
			&=\bt_{g_1}\Big(\f\Big(1_{g\m_1}\Big)\f(1_{(g_2,\dots,g_{n+1})})u(g_2,\dots,g_{n+1})\Big)\\
			&\quad +\sum_{i=1}^n (-1)^i \f(1_{g_1\dots g_i})\f(1_{(g_1,\dots,g_ig_{i+1},\dots,g_{n+1})})u(g_1,\dots,g_ig_{i+1},\dots,g_{n+1})\\
			&\quad + (-1)^{n+1} \f(1_{g_1\dots g_{n+1}})\f(1_{(g_1,\dots,g_n)})u(g_1,\dots,g_n)\\
			&=1_{(g_1,\dots,g_{n+1})}(\bt^*_{g_1}(u(g_2,\dots,g_{n+1}))\\
			&\quad +\sum_{i=1}^n (-1)^i u(g_1,\dots,g_ig_{i+1},\dots,g_{n+1})\\
			&\quad + (-1)^{n+1} u(g_1,\dots,g_n))\\
			&=1_{(g_1,\dots,g_{n+1})}(\dl^nu)(g_1,\dots,g_{n+1})\\
			&=\f(\rho(\dl^nu)(g_1,\dots,g_{n+1})),
		\end{align*}
		so that $\dl^n\rho(u)=\rho(\dl^nu)$.
	\end{proof}
	
	\begin{defn}\label{glob-coc-defn}
		Given $w\in Z^n_{par}(G,\A)$, by a {\it globalization} of $w$ we mean $u\in Z^n(G,\M(\B))$ satisfying \cref{w-is-restr-of-u}. If $w$ admits a globalization, then we say that $w$ is {\it globalizable}.
	\end{defn}
	
	Recall that the enveloping action $(\B,\bt)$ for $(\A,\af)$ was constructed in~\cite{DE} as the restriction of the global action $(\cF,\bt)$ to the subalgebra 
	\begin{align}\label{B=sum-bt_g(f(A))}
		\B=\sum_{g\in G}\bt_g(\f(\A)).
	\end{align}
	Here $\cF$ is the ring of functions $G\to\A$ and 
	\begin{align}\label{beta_x(f)_t}
		\beta_g(f)|_t=f(g\m t) 
	\end{align}
	for all $x,t\in G$, where the notation $f|_t$ from~\cite{DE} is used for the value $f(t)$ of $f\in\cF$ at $t\in G$. The injective morphism $\f:\A\to\cF$ is then defined by the formula
	\begin{align}\label{f(a)|_t=af_t-inv(1_ta)}
		\f(a)|_t=\af_{t\m}(1_ta).
	\end{align} 
	Clearly, $\f(\A)\sst\B$, so $\f$ is a morphism $(\A,\af)\to(\B,\bt)$ too. Since all enveloping actions of $(\A,\af)$ are isomorphic~\cite{DE} to each other, we may always assume that $(\B,\bt)$ and $\f$ are of this form.
	
	\begin{lem}\label{0-cocycle-is-globalizable}
		Any $w\in Z^0_{par}(G,\A)$ is uniquely globalizable.
	\end{lem}
	\begin{proof}
		Define $u\in C^0(G,\cF)=\cF$ to be the constant function taking the value $w\in\A$ at any $t\in G$. Using \cref{f(a)|_t=af_t-inv(1_ta),(dl^0a)(g)=af_g(1_(g-inv)a)-1_ga}, we obtain
		\begin{align*}
			\f(1_\A)|_tu|_t=1_{t\m}w=\af_{t\m}(1_tw)=\f(w)|_t,
		\end{align*}
		yielding \cref{w-is-restr-of-u}. The proof of the formula $\bt_g(\f(a))|_tu|_t=\bt_g(\f(aw))|_t$ from~\cite[Remark~2.3]{DKS} works here without any change, so $\bt_g(\f(\A))u\sst \f(\A)$. Hence $\B u\sst \B$ by \cref{B=sum-bt_g(f(A))}. In a similar way $u|_t\bt_g(\f(a))|_t=\bt_g(\f(wa))|_t$, which implies $u\B \sst \B$, and thus $u\in C^0(G,\M(\B))$. To prove the $0$-cocycle identity $\bt^*_g(u)=u$ for $u$, it suffices to show that $\bt_g(uf)=u\bt_g(f)$ for any $f\in\cF$. We have by \cref{beta_x(f)_t}
		\begin{align*}
			\bt_g(uf)|_t=(uf)|_{g\m t}=u|_{g\m t}f|_{g\m t}=u|_t\bt_g(f)|_t,
		\end{align*}
		whence $u\in Z^0(G,\M(\B))$.
		
		The uniqueness of $u$ is proved the same way as in~\cite[Remark~2.3]{DKS}.
	\end{proof}
	
	For the case $w\in Z^n_{par}(G,\A)$, $n\in\bbN$, we shall need an ``additive'' version of~\cite[Lemma~2.1]{DKS}.
	\begin{lem}\label{u-is-n-cocycle}
		Let $\wtl w\in C^n(G,\A)$. Then $u\in C^n(G,\cF)$, defined by
		\begin{align}\label{u-def-n>0}
			u(g_1,\dots,g_n)|_t&=(-1)^n\wtl w(t\m, g_1,\dots,g_{n-1})+\wtl w(t\m g_1, g_2, \dots,g_n)\notag\\
			&\quad+\sum_{i=1}^{n-1}(-1)^i\wtl w(t\m,g_1,\dots,g_ig_{i+1},\dots,g_n),
		\end{align}
		is an $n$-cocycle with respect to the action $\beta$ of $G$ on $\cF$.
	\end{lem}
	\begin{proof}
		Observe by \cref{u-def-n>0} that
		\begin{align}\label{u-in-terms-of-tilde-delta}
			u(g_1,\dots,g_n)|_t=\wtl w(g_1,\dots,g_n)-(\tl\dl^n\wtl w)(t\m,g_1,\dots,g_n),
		\end{align}
		where $\tl\dl^n:C^n(G,\A)\to C^{n+1}(G,\A)$ is the coboundary operator which corresponds to the trivial $G$-module, i.e.
		\begin{align*}
			(\tl\dl^n\wtl w)(g_1,\dots,g_{n+1})&=\wtl w(g_2,\dots,g_{n+1})\\
			&\quad+\sum_{i=1}^n(-1)^i\wtl w(g_1,\dots,g_ig_{i+1},\dots,g_{n+1})\\
			&\quad+(-1)^{n+1}\wtl w(g_1,\dots,g_n).
		\end{align*}
		Calculating the value of $(\dl^nu)(g_1,\dots,g_{n+1})$ at $t\in G$, we obtain using \cref{beta_x(f)_t}
		\begin{align*}
			u(g_2,\dots,g_{n+1})|_{g\m_1 t}&+\sum_{i=1}^n (-1)^iu(g_1,\dots,g_ig_{i+1},\dots,g_{n+1})|_t\\
			&+(-1)^{n+1}u(g_1,\dots,g_n)|_t,
		\end{align*}
		which in view of \cref{u-in-terms-of-tilde-delta} equals
		\begin{align*}
			&\wtl w(g_2,\dots,g_{n+1})-(\tl\dl^n\wtl w)(t\m g_1,g_2,\dots,g_{n+1})\\
			&+\sum_{i=1}^n (-1)^i\wtl w(g_1,\dots,g_ig_{i+1},\dots,g_{n+1})\\
			&+\sum_{i=1}^n(-1)^{i+1}(\tl\dl^n\wtl w)(t\m,g_1,\dots,g_ig_{i+1},\dots,g_{n+1})\\
			&+(-1)^{n+1}\wtl w(g_1,\dots,g_n)+(-1)^n(\tl\dl^n\wtl w)(t\m,g_1,\dots,g_n).
		\end{align*}
		The latter is readily seen to be $(\tl\dl^{n+1}\tl\dl^n\wtl w)(t\m,g_1,\dots,g_{n+1})=0_\A$.
	\end{proof}
	
	As in~\cite[Theorem~2.4]{DKS}, the existence of a globalization of $w\in Z^n_{par}(G,\A)$ is equivalent to certain extendibility property. For any $f\in C^n(G,\A)$ define
	\begin{align} 
		(\tl\dl^n f)(g_1,\dots,g_{n+1})&=\af_{g_1}\left(1_{g\m_1}f(g_2,\dots,g_{n+1})\right)\notag\\  
		&\quad+\sum_{i=1}^n(-1)^i 1_{g_1}f(g_1,\dots,g_i g_{i+1},\dots,g_{n+1})\notag\\ 
		&\quad+(-1)^{n+1} 1_{g_1}f(g_1,\dots,g_n).\label{(tl-dl^n-f)(g_1...g_n)}
	\end{align}
	
	\begin{thrm}\label{w-glob-iff-exists-tilde-w}
		A cocycle $w\in Z^n_{par}(G,\A)$, $n\in\bbN$, is globalizable if and only if there exists $\wtl w\in C^n(G,\A)$ such that 
		\begin{align}\label{cocycle-ident-for-tilde-w}
			\tl\dl^n\wtl w=0
		\end{align} 
		and 
		\begin{align}\label{w-is-restr-of-tilde-w}
			w(g_1,\dots,g_n)=1_{(g_1,\dots,g_n)}\wtl w(g_1,\dots,g_n), 
		\end{align} 
		for all $g_1,\dots,g_n\in G$.
	\end{thrm}
	\begin{proof}
		If $w$ is globalizable and $u\in Z^n(G,\M(\B))$ is its globalization, then as in the proof of~\cite[Theorem~2.4]{DKS} we define
		\begin{align*}
			\f(\wtl w(g_1,\dots,g_n))=\f(1_\A)u(g_1,\dots,g_n)=u(g_1,\dots,g_n)\f(1_\A).
		\end{align*}
		Clearly, $w(g_1,\dots,g_n)\in\A$, since $\f(\A)$ is an ideal in $\B$, and moreover  \cref{w-is-restr-of-tilde-w} is satisfied. Using the formula
		\begin{align*}
			\bt^*_{g_1}(u(g_2,\dots,g_{n+1}))\f(1_{g_1})=\f\left(\af_{g_1}\left( 1_{g\m_1}\wtl w(g_2,\dots,g_{n+1})\right)\right),
		\end{align*}
		which follows from \cref{w-is-restr-of-tilde-w} as in the proof of~\cite[Theorem~2.4]{DKS}, we obtain \cref{cocycle-ident-for-tilde-w} by applying both sides of the cocycle identity
		\begin{align*}
			\bt^*_{g_1}(u(g_2,\dots,g_{n+1}))
			&+\sum_{i=1}^n (-1)^iu(g_1,\dots,g_ig_{i+1},\dots,g_{n+1})\\
			&+(-1)^{n+1}u(g_1,\dots,g_n)=0
		\end{align*}
		to $\f(1_{g_1})$.
		
		Conversely, given $\wtl w\in C^n(G,\A)$ satisfying \cref{cocycle-ident-for-tilde-w,w-is-restr-of-tilde-w}, define $u\in C^n(G,\cF)$ by \cref{u-def-n>0}. We immediately have $u\in Z^n(G,\cF)$ by \cref{u-is-n-cocycle}. Now, using \cref{f(a)|_t=af_t-inv(1_ta),w-is-restr-of-tilde-w,u-def-n>0} and the cocycle identity for $w$, we obtain
		\begin{align*}
			\f(w(g_1,\dots,g_n))|_t&=\af_{t\m}(1_tw(g_1,\dots,g_n))\\
			&=1_{t\m}w(t\m g_1, g_2,\dots,g_n)\\
			&\quad+\sum_{i=1}^{n-1}(-1)^i1_{t\m g_1\dots g_i}w(t\m, g_1,\dots,g_ig_{i+1},\dots,g_n)\\
			&\quad+(-1)^n1_{t\m g_1\dots g_n}w(t\m,g_1,\dots,g_{n-1})\\
			&=1_{t\m}1_{(t\m g_1, g_2,\dots,g_n)}\wtl w(t\m g_1, g_2,\dots,g_n)\\
			&\quad+\sum_{i=1}^{n-1}(-1)^i1_{t\m g_1\dots g_i}1_{(t\m, g_1,\dots,g_ig_{i+1},\dots,g_n)}\wtl w(t\m, g_1,\dots,g_ig_{i+1},\dots,g_n)\\
			&\quad+(-1)^n1_{t\m g_1\dots g_n}1_{(t\m,g_1,\dots,g_{n-1})}\wtl w(t\m,g_1,\dots,g_{n-1})\\
			&=1_{(t\m,g_1,\dots,g_n)}u(g_1,\dots,g_n)|_t\\
			&=\f(1_{(g_1,\dots,g_n)})|_tu(g_1,\dots,g_n)|_t,
		\end{align*}
		whence \cref{w-is-restr-of-u}.
		
		We have yet to prove that $u(g_1,\dots,g_n)\in\M(\B)$, i.e. 
		\begin{align}\label{u(g_1...g_n)B-sst-B}
			u(g_1,\dots,g_n)\B,\B u(g_1,\dots,g_n)\sst\B
		\end{align}
		for all $g_1,\dots,g_n\in G$. Since by \cref{cocycle-ident-for-tilde-w,u-def-n>0}
		\begin{align*}  
			1_{t\m}u(g_1,\dots,g_n)|_t&=1_{t\m}\wtl w(t\m g_1, g_2,\dots,g_n)\\ 
			&\quad+\sum_{i=1}^{n-1}(-1)^i1_{t\m}\wtl w(t\m,g_1,\dots,g_i g_{i+1},\dots,g_n)\\
			&\quad+(-1)^n1_{t\m}\wtl w(t\m,g_1,\dots,g_{n-1})\\
			&=\af_{t\m}(1_t\wtl w(g_1,\dots,g_n)),
		\end{align*} 
		it follows that
		\begin{align*}
			u(g_1,\dots,g_n)|_t\f(a)|_t=\af_{t\m}(1_t\wtl w(g_1,\dots,g_n))\af_{t\m}(1_ta)=\f(\wtl w(g_1,\dots,g_n)a)|_t,
		\end{align*} 
		whence 
		\begin{align}\label{u(g_1...g_n)f(A)-sst-f(A)}
			u(g_1,\dots,g_n)\f(\A)\sst\f(\A).
		\end{align}
		Now $u$, being an $n$-cocycle with values in $(\cF,\bt)$, satisfies
		\begin{align*} 
			\bt_{t\m}(u(g_1,\dots,g_n))\f(a)&=u(t\m g_1, g_2,\dots,g_n)\f(a)\\ 
			&\quad+\sum_{i=1}^{n-1}(-1)^iu(t\m,g_1,\dots,g_ig_{i+1},\dots,g_{n+1})\f(a)\\
			&\quad+(-1)^nu(t\m, g_1,\dots,g_{n-1})\f(a),
		\end{align*}
		where the right-hand side is an element of $\f(\A)$ thanks to \cref{u(g_1...g_n)f(A)-sst-f(A)}. Therefore, $\bt_{t\m}(u(g_1,\dots,g_n))\f(\A)\sst\f(\A)$, so, applying $\bt_t$, we obtain $u(g_1,\dots,g_n)\bt_t(\f(\A))\sst\bt_t(\f(\A))$. Similarly, $\bt_t(\f(\A))u(g_1,\dots,g_n)\sst\bt_t(\f(\A))$, proving \cref{u(g_1...g_n)B-sst-B} in view of \cref{B=sum-bt_g(f(A))}.
	\end{proof}
	
	\subsection{The construction of $w'$}\label{sec:w'}
	
	From now on we assume that $\A=\prod_{\lb\in\Lb}\A_\lb$, where each $\A_\lb$ is an indecomposable unital ring, called a \emph{block} of $\A$. Our aim is to show that every $w\in Z^n_{par}(G,\A)$ can be replaced by a more manageable $w'\in Z^n_{par}(G,\A)$ which will be used in the construction of $\wtl w$ satisfying the conditions of \cref{w-glob-iff-exists-tilde-w}.
	
	As in~\cite{DKS}, the identity $1_{\A_\lb}$ of $\A_{\lb}$ will be identified with an (indecomposable) central idempotent of $\A$, the block $\A_\lb$ with the ideal $1_{\A_\lb}\A$ of $\A$, and the canonical projection $\pr_\lb:\A\to\A_\lb$ with the multiplication by $1_{\A_\lb}$ in $\A$. We write $a=\prod_{\lb\in\Lb_1}a_\lb$, where $\Lb_1\sst\Lb$ and $a_\lb\in\A_\lb$ for all $\lb\in\Lb_1$, if
	\begin{align*}
		\pr_\lb(a)=
		\begin{cases}
			a_\lb, & \lb\in\Lb_1,\\
			0_\A,  & \mbox{otherwise}.
		\end{cases}
	\end{align*}
	Thus, each idempotent $e$ of $\A$ is central and is of the form $\prod_{\lb\in\Lb_1}1_{\A_\lb}$, so that $e\A=\prod_{\lb\in\Lb_1}\A_\lb$. Moreover, an isomorphism between two unital ideals $e\A$ and $f\A$ maps a block of $e\A$ onto a block of $f\A$ (see~\cite[Lemma~3.1]{DKS}).
	
	A unital partial action $\af$ of a group $G$ on $\A$ is called \emph{transitive}, if for all $\lb',\lb''\in\Lb$ there exists $x\in G$, such that $\A_{\lb'}\sst\cD_{x\m}$ and $\af_x(\A_{\lb'})=\A_{\lb''}$. As in~\cite{DKS}, we fix $\lb_0\in\Lb$ and denote by $H$ the \emph{stabilizer} of the block $\A_{\lb_0}$, i.e. the subgroup
	\begin{align*}
		H=\{x\in G\mid \A_{\lb_0}\sst\cD_{x\m}\mbox{ and }\alpha_x(\A_{\lb_0})=\A_{\lb_0}\}
	\end{align*}
	of $G$. Let $\Lb'$ be a left transversal of $H$ in $G$ containing the identity element $1$ of $G$. Then $\Lb$ can be identified with a subset of $\Lb'$, namely, $\lb_0$ is identified with $1$ and
	\begin{align*}
		\A_g=\alpha_g(\A_1)\mbox{ for }g\in\Lb\sst\Lb'.
	\end{align*}
	Given $x\in G$, denote by $\bar x$ the (unique) element of $\Lb'$, such that $x\in\bar xH$. We shall use the following easy fact throughout the text.
	\begin{lem}[Lemma 5.1 from~\cite{DES2}]\label{Lb-and-Lb'}
		Given $x\in G$ and $g\in\Lb'$, one has
		\begin{enumerate}
			\item $g\in\Lb\iff \A_1\sst\cD_{g\m}$;\label{g-in-Lb-iff-A_1-in-D_g-inv}
			\item if $g\in\Lb$, then $\ol{xg}\in\Lb\iff \A_g\sst\cD_{x\m}$, in which case $\af_x(\A_g)=\A_{\ol{xg}}$.\label{ol-xg-in-Lb-iff-A_g-in-D_x-inv}
		\end{enumerate}
	\end{lem} 
	It follows that
	\begin{align}\label{A_g-sst-D_x-iff-A_1-sst-D_g-inv-x}
		\A_g\sst\cD_x\iff\ol{x\m g}\in\Lb\iff\A_1\sst\cD_{g\m x}.
	\end{align}
	In particular, 
	\begin{align}\label{A_(bar-x)-sst-D_x}
		\A_{\ol{x}}\sst\cD_x
	\end{align}
	for all $x\in G$, such that $\ol{x}\in\Lb$.
	
	As in~\cite{DKS}, the definition of $w'$ will involve the homomorphism $\0_g:\A\to\A_g$ given by
	\begin{align}\label{0_g-def}
		\0_g(a)=\af_g(\pr_1(a))=\pr_g(\af_g(1_{g\m}a)),
	\end{align}
	where $g\in\Lb$ and $a\in\A$. It follows that 
	\begin{align}\label{a=prod-0_g-circ-af_g-inv}
		a=\prod_{g\in\Lb}\0_g(\alpha_{g\m}(1_ga))
	\end{align}
	(see formula~(31) from~\cite{DKS}). Another fact that we shall use:
	\begin{align}\label{0_g(a)=0_g(1_xa)}
		\0_g(a)=\0_g(1_xa) 
	\end{align}
	for any $x\in G$, such that $\A_1\sst\cD_x$. In particular, this holds for $x\in H$ and for $x=g\m$.
	
	\begin{lem}\label{w=prod-0_g}
		Let $n>0$ and $w\in Z_{par}^n(G,\A)$. Then
		\begin{align}
			w(x_1,\dots,x_n)&=1_{(x_1,\dots,x_n)}\prod_{g\in\Lb}\0_g[w(g\m x_1,x_2,\dots,x_n)\notag\\
			&\quad +\sum_{k=1}^{n-1}(-1)^k w(g\m,x_1,\dots,x_kx_{k+1},\dots,x_n)\notag\\
			&\quad +(-1)^n w(g\m,x_1,\dots,x_{n-1})].\label{w=prod_g-0_g-of-w}
		\end{align}
	\end{lem}
	\begin{proof}
		By \cref{a=prod-0_g-circ-af_g-inv,(dl^nf)(g_1...g_(n+1))}
		\begin{align*}
			w(x_1,\dots,x_n)&=\prod_{g\in\Lb}\0_g(\alpha_{g\m}(1_gw(x_1,\dots,x_n)))\\
			&=\prod_{g\in\Lb}\0_g[1_{g\m}w(g\m x_1,x_2,\dots,x_n)\\
			&\quad +\sum_{k=1}^{n-1}(-1)^k 1_{g\m x_1\dots x_k}w(g\m,x_1,\dots,x_kx_{k+1},\dots,x_n)\\
			&\quad +(-1)^n 1_{g\m x_1\dots x_n}w(g\m,x_1,\dots,x_{n-1})]\\
			&=\prod_{g\in\Lb}\0_g[1_{(g\m,x_1,\dots,x_n)}(w(g\m x_1,x_2,\dots,x_n)\\
			&\quad +\sum_{k=1}^{n-1}(-1)^k w(g\m,x_1,\dots,x_kx_{k+1},\dots,x_n)\\
			&\quad +(-1)^n w(g\m,x_1,\dots,x_{n-1}))].
		\end{align*}
		It remains to observe by \cref{0_g-def,A_(bar-x)-sst-D_x} that
		\begin{align*}
			\0_g(1_{(g\m,x_1,\dots,x_n)})=\pr_g(\af_g(1_{g\m}1_{(g\m,x_1,\dots,x_n)}))=\pr_g(1_g1_{(x_1,\dots,x_n)})=\pr_g(1_{(x_1,\dots,x_n)}),
		\end{align*}
		so that 
		\begin{align}\label{prod-0_g(1_(g-inv_x_1...x_n))}
			\prod_{g\in\Lb}\0_g(1_{(g\m,x_1,\dots,x_n)})=\prod_{g\in\Lb}\pr_g(1_{(x_1,\dots,x_n)})=1_{(x_1,\dots,x_n)}.
		\end{align}
	\end{proof}
	
	We recall here the notations from~\cite{DKS}. We denote by $\eta$ the map $G\to H$ which sends $x\in G$ to $x\m\bar x\in H$. For all $n>0$ and $g\in\Lb'$ we define $\eta_n^g:G^n\to H$ by
	\begin{align}\label{eta_n^g-def}
		\eta_n^g(x_1,\dots,x_n)=\eta(x\m_n \ol{x\m_{n-1}\dots x\m_1 g})
	\end{align}
	and $\tau_n^g:G^n\to H^n$ by
	\begin{align}\label{tau_n^g-def}
		\tau_n^g(x_1,\dots,x_n)=(\eta_1^g(x_1),\eta_2^g(x_1,x_2),\dots,\eta_n^g(x_1,\dots,x_n)). 
	\end{align}
	We notice here that 
	\begin{align}\label{prod-of-etas}
		\eta_1^g(x_1)\eta_2^g(x_1,x_2)\dots\eta_n^g(x_1,\dots,x_n)=\eta(x\m_n\dots x\m_1g)=\eta_1^g(x_1\dots x_n).
	\end{align}
	Furthermore, the functions $\s_{n,i}^g:G^n\to G^{n+1}$, $n>0$, $0\le i\le n$, will be defined by
	\begin{align}
		\s_{n,0}^g(x_1,\dots,x_n)&=(g\m,x_1,\dots,x_n),\label{sigma_n0-def}\\
		\s_{n,i}^g(x_1,\dots,x_n)&=(\tau_i^g(x_1,\dots,x_i),(\ol{x\m_i\dots x\m_1 g})\m,x_{i+1},\dots,x_n),\ \ 0<i<n,\label{sigma_ni-def}\\
		\s_{n,n}^g(x_1,\dots,x_n)&=(\tau_{n}^g(x_1,\dots,x_{n}),(\ol{x\m_n\dots x\m_1 g})\m).\label{sigma_nn-def}
	\end{align}
	If $n=0$, then we set 
	\begin{align}\label{sigma_00-def}
		\s_{0,0}^g=g\m\in G.
	\end{align}
	
	\begin{defn}\label{w'-and-eps-defn}
		Given $n>0$ and $w\in C_{par}^n(G,\A)$, define $w'\in C^n_{par}(G,\A)$ and $\ve\in C^{n-1}_{par}(G,\A)$ by
		\begin{align}
			w'(x_1,\dots,x_n)&=1_{(x_1,\dots,x_n)}\prod_{g\in\Lb}\0_g\circ w\circ\tau_n^g(x_1,\dots,x_n),\label{w'-def}\\
			\ve(x_1,\dots,x_{n-1})&=1_{(x_1,\dots,x_{n-1})}\prod_{g\in\Lb}\0_g\left(\sum_{i=0}^{n-1}(-1)^i w\circ\s_{n-1,i}^g(x_1,\dots,x_{n-1})\right).\label{eps-def}
		\end{align}
		When $n=1$, equality \cref{eps-def} should be understood as 
		\begin{align}\label{eps-n=0-def}
			\ve=\prod_{g\in\Lb}\0_g(w(g\m))\in\A.
		\end{align}
	\end{defn}
	
	We introduce here an additive analogue of the notation used in~\cite{DKS}:
	\begin{align}
		\Sg(l,m)&=\sum_{k=l,i=m}^{n-1}(-1)^{k+i} w\circ\s_{n-1,i}^g(x_1,\dots,x_kx_{k+1},\dots,x_n)\notag\\
		&\quad+\sum_{i=m}^{n-1}(-1)^{n+i} w\circ\s_{n-1,i}^g(x_1,\dots,x_{n-1}),\label{Sg(l_m)-def}
	\end{align}
	where $1\le l\le n-1$ and $0\le m\le n-1$ ($n$ is assumed to be fixed). 
	
	\begin{lem}\label{w'-cohom-w-base-0}
		For all $w\in Z_{par}^1(G,\A)$ and  $x\in G$ we have:
		\begin{align}\label{delta-e-alpha-inv-w-inv-n=1}
			(\dl^0\ve)(x)-\af_x(1_{x\m}\ve)-w(x)=1_x\prod_{g\in\Lb}\0_g(-w(g\m x)).  
		\end{align}
		Moreover, for $n>1$, $w\in Z^n(G,\A)$ and $x_1,\dots,x_n\in G$:
		\begin{align}
			&(\dl^{n-1}\ve)(x_1,\dots,x_n)-\af_{x_1}(1_{x\m_1}\ve(x_2,\dots,x_n))-w(x_1,\dots,x_n)\notag\\
			&=1_{(x_1,\dots,x_n)}\prod_{g\in\Lb}\0_g(-w(g\m x_1,x_2,\dots,x_n)+\Sg(1,1)).\label{delta-e-alpha-inv-w-inv}
		\end{align}
	\end{lem}
	\begin{proof}
		By \cref{(dl^0a)(g)=af_g(1_(g-inv)a)-1_ga,eps-n=0-def,w=prod_g-0_g-of-w} the left-hand side of \cref{delta-e-alpha-inv-w-inv-n=1} equals
		\begin{align*}
			-1_x\ve-w(x)&=-1_x\prod_{g\in\Lb}\0_g(w(g\m))-1_x\prod_{g\in\Lb}\0_g(w(g\m x)-w(g\m)),
		\end{align*}
		proving \cref{delta-e-alpha-inv-w-inv-n=1}.
		
		Now by \cref{(dl^nf)(g_1...g_(n+1)),Sg(l_m)-def,eps-def}
		\begin{align*}
			&(\dl^{n-1}\ve)(x_1,\dots,x_n)-\af_{x_1}(1_{x\m_1}\ve(x_2,\dots,x_n))\\
			&=\sum_{k=1}^{n-1}(-1)^k 1_{x_1\dots x_k}\ve(x_1,\dots,x_kx_{k+1},\dots,x_n)+(-1)^n1_{x_1\dots x_n}\ve(x_1,\dots,x_{n-1})\\
			&=1_{(x_1,\dots,x_n)}\prod_{g\in\Lb}\0_g\left(\sum_{k=1,i=0}^{n-1}(-1)^{k+i} w\circ\s_{n-1,i}^g(x_1,\dots,x_kx_{k+1},\dots,x_n)\right)\\
			&+1_{(x_1,\dots,x_n)}\prod_{g\in\Lb}\0_g\left(\sum_{i=0}^{n-1}(-1)^{n+i} w\circ\s_{n-1,i}^g(x_1,\dots,x_{n-1})\right)\\
			&=1_{(x_1,\dots,x_n)}\prod_{g\in\Lb}\0_g(\Sg(1,0)).
		\end{align*}
		As in the proof of~\cite[Lemma~3.5]{DKS}, using \cref{w=prod_g-0_g-of-w} we conclude that the latter is
		\begin{align*}
			w(x_1,\dots,x_n)+1_{(x_1,\dots,x_n)}\prod_{g\in\Lb}\0_g(-w(g\m x_1,x_2,\dots,x_n)+\Sg(1,1)).
		\end{align*}
	\end{proof}
	
	\begin{lem}\label{w'-cohom-w-base-1}
		For all $n>1$, $w\in Z_{par}^n(G,\A)$, $g\in\Lb$ and $x_1,\dots,x_n\in G$:
		\begin{align}
			&1_{\s_{n,1}^g(x_1,\dots,x_n)}(-w(g\m x_1,x_2,\dots,x_n)+\Sg(1,1))\notag\\
			&\quad=-\alpha_{\eta_1^g(x_1)}(1_{\eta_1^g(x_1)\m}w\circ\s_{n-1,0}^{\ol{x\m_1 g}}(x_2,\dots,x_n))\notag\\
			&\quad\quad +1_{\s_{n,1}^g(x_1,\dots,x_n)}(-w(\tau_1^g(x_1),(\ol{x\m_1 g})\m x_2,x_3,\dots,x_n)+\Sg(2,2)\notag\\
			&\quad\quad+\sum_{i=1}^{n-1}(-1)^{i+1} w\circ\s_{n-1,i}^g(x_1x_2,x_3,\dots,x_n)).\label{-w+Sg(1_1)}
		\end{align}
	\end{lem}
	\begin{proof}
		By \cref{eta_n^g-def,tau_n^g-def,sigma_ni-def,(dl^nf)(g_1...g_n)=...}
		\begin{align*}
			0&=(\dl^nw)\circ\s_{n,1}^g(x_1,\dots,x_n)\\
			&=(\dl^nw)(g\m x_1\cdot\ol{x\m_1 g},(\ol{x\m_1 g})\m,x_2,\dots,x_n)\\
			&=\alpha_{g\m x_1\cdot\ol{x\m_1 g}}(1_{{(\ol{x\m_1 g})\m x\m_1 g}}w((\ol{x\m_1 g})\m,x_2,\dots,x_n))\\
			&\quad-1_{g\m x_1\cdot\ol{x\m_1 g}}w(g\m x_1,x_2,\dots,x_n)\\
			&\quad+1_{g\m x_1}w(g\m x_1\cdot\ol{x\m_1 g},(\ol{x\m_1 g})\m x_2,x_3,\dots,x_n)\\
			&\quad+\sum_{k=2}^{n-1} (-1)^{k+1} 1_{g\m x_1\dots x_k} w(g\m x_1\cdot\ol{x\m_1 g},(\ol{x\m_1 g})\m,x_2,\dots,x_kx_{k+1},\dots x_n)\\
			&\quad +(-1)^{n+1} 1_{g\m x_1\dots x_n} w(g\m x_1\cdot\ol{x\m_1 g},(\ol{x\m_1 g})\m,x_2,\dots,x_{n-1}).
		\end{align*}
		Therefore,
		\begin{align*}
			-1_{\eta_1^g(x_1)}w(g\m x_1,x_2,\dots,x_n)&=-\alpha_{\eta_1^g(x_1)}(1_{\eta_1^g(x_1)\m}w\circ\s_{n-1,0}^{\ol{x\m_1 g}}(x_2,\dots,x_n))\\
			&\quad -1_{g\m x_1}w(\tau_1^g(x_1),(\ol{x\m_1 g})\m x_2,x_3,\dots,x_n)\\
			&\quad +\sum_{k=2}^{n-1} (-1)^k 1_{g\m x_1\dots x_k} w\circ\s_{n-1,1}^g(x_1,\dots,x_kx_{k+1},\dots x_n)\\
			&\quad +(-1)^n 1_{g\m x_1\dots x_n} w\circ\s_{n-1,1}^g(x_1,\dots,x_{n-1}).
		\end{align*}
		Adding $\Sg(1,1)$ and then multiplying both sides of the obtained equality by $1_{\s_{n,1}^g(x_1,\dots,x_n)}=1_{\eta_1^g(x_1)}1_{(g\m x_1,x_2,\dots,x_n)}$, we get \cref{-w+Sg(1_1)} (for more details in the multiplicative case see the proof of~\cite[Lemma~3.6]{DKS}).
	\end{proof}
	
	\begin{lem}\label{w'-cohom-w-step}
		For all $1<j<n$, $w\in Z_{par}^n(G,\A)$, $g\in\Lb$ and $x_1,\dots,x_n\in G$:
		\begin{align}
			&1_{\s_{n,j}^g(x_1,\dots,x_n)}(-w(\tau_{j-1}^g(x_1,\dots,x_{j-1}),(\ol{x\m_{j-1}\dots x\m_1 g})\m x_j,x_{j+1},\dots,x_n)+\Sg(j,j))\notag\\
			&\quad=(-1)^j\af_{\eta_1^g(x_1)}(1_{\eta_1^g(x_1)\m}w\circ\s_{n-1,j-1}^{\ol{x\m_1 g}}(x_2,\dots,x_n))\notag\\
			&\quad\quad +1_{\s_{n,j}^g(x_1,\dots,x_n)}(-w(\tau_j^g(x_1,\dots,x_j),(\ol{x\m_j\dots x\m_1 g})\m x_{j+1},x_{j+2},\dots,x_n)+\Sg(j+1,j+1)\notag\\
			&\quad\quad +\sum_{i=j}^{n-1} (-1)^{i+j} w\circ\s_{n-1,i}^g(x_1,\dots,x_jx_{j+1},\dots,x_n)\notag\\
			&\quad\quad +\sum_{s=1}^{j-1} (-1)^{s+j} w\circ\s_{n-1,j-1}^g(x_1,\dots,x_sx_{s+1},\dots,x_n))\label{w-Sg(j_j)}
		\end{align}
		(here by $\Sg(n,n)$ we mean $0_\A$).
	\end{lem}
	\begin{proof}
		Our argument is analogous to that of the proof of \cref{w'-cohom-w-base-1} (for some technical details see also the proof of \cite[Lemma~3.7]{DKS}):
		\begin{align*}
			0&=(\dl^nw)\circ\s_{n,j}^g(x_1,\dots,x_n)\\
			&=\af_{\eta_1^g(x_1)}(1_{\eta_1^g(x_1)\m}w\circ\s_{n-1,j-1}^{\ol{x\m_1 g}}(x_2,\dots,x_n))\\
			&\quad+\sum_{s=1}^{j-1} (-1)^s 1_{\eta_1^g(x_1\dots x_s)}w\circ\s_{n-1,j-1}^g(x_1,\dots,x_sx_{s+1},\dots,x_n)\\
			&\quad +(-1)^j 1_{\eta_1^g(x_1\dots x_j)} w(\tau_{j-1}^g(x_1,\dots,x_{j-1}),(\ol{x\m_{j-1}\dots x\m_1g})\m x_j,x_{j+1},\dots,x_n)\\
			&\quad +(-1)^{j+1} 1_{g\m x_1\dots x_j} w(\tau_j^g(x_1,\dots,x_j),(\ol{x\m_j\dots x\m_1g})\m x_{j+1},x_{j+2},\dots,x_n)\\
			&\quad+\sum_{t=j+1}^{n-1} (-1)^{t+1} 1_{g\m x_1\dots x_t} w\circ\s_{n-1,j}^g(x_1,\dots,x_tx_{t+1},\dots,x_n)\\
			&\quad +(-1)^{n+1} 1_{g\m x_1\dots x_n} w\circ\s_{n-1,j}^g(x_1,\dots,x_{n-1}).
		\end{align*}
		It follows that 
		\begin{align*}
			&-1_{\eta_1^g(x_1\dots x_j)} w(\tau_{j-1}^g(x_1,\dots,x_{j-1}),(\ol{x\m_{j-1}\dots x\m_1g})\m x_j,x_{j+1},\dots,x_n)\\
			&\quad=(-1)^j\af_{\eta_1^g(x_1)}(1_{\eta_1^g(x_1)\m}w\circ\s_{n-1,j-1}^{\ol{x\m_1 g}}(x_2,\dots,x_n))\\
			&\quad\quad+\sum_{s=1}^{j-1} (-1)^{s+j} 1_{\eta_1^g(x_1\dots x_s)}w\circ\s_{n-1,j-1}^g(x_1,\dots,x_sx_{s+1},\dots,x_n)\\
			&\quad\quad -1_{g\m x_1\dots x_j} w(\tau_j^g(x_1,\dots,x_j),(\ol{x\m_j\dots x\m_1g})\m x_{j+1},x_{j+2},\dots,x_n)\\
			&\quad\quad+\sum_{t=j+1}^{n-1} (-1)^{t+j+1} 1_{g\m x_1\dots x_t} w\circ\s_{n-1,j}^g(x_1,\dots,x_tx_{t+1},\dots,x_n)\\
			&\quad\quad +(-1)^{n+j+1} 1_{g\m x_1\dots x_n} w\circ\s_{n-1,j}^g(x_1,\dots,x_{n-1}).
		\end{align*}
		The addition of $\Sg(j,j)$ followed by the multiplication of both sides by the idempotent
		\begin{align*}
			1_{\s_{n,j}^g(x_1,\dots,x_n)}=1_{\eta_1^g(x_1)}1_{\eta_1^g(x_1x_2)}\dots 1_{\eta_1^g(x_1\dots x_j)}1_{(g\m x_1\dots x_j,x_{j+1},\dots,x_n)}
		\end{align*}
		gives the desired equality \cref{w-Sg(j_j)}.
	\end{proof}
	
	\begin{lem}\label{w'-cohom-w-final-step}
		For all $w\in Z_{par}^1(G,\A)$, $g\in\Lb$ and $x\in G$:
		\begin{align}\label{w-Sg(n_n)-n=1}
			-1_{\eta_1^g(x)}w(g\m x)=-\af_{\eta_1^g(x)}(1_{\eta_1^g(x)\m}w((\ol{x\m g})\m))-1_{g\m x}w\circ\tau_1^g(x).
		\end{align}
		Moreover, for all $n>1$, $w\in Z_{par}^n(G,\A)$, $g\in\Lb$ and $x_1,\dots,x_n\in G$:
		\begin{align}
			&-1_{\eta_1^g(x_1\dots x_n)}w(\tau_{n-1}^g(x_1,\dots,x_{n-1}),(\ol{x\m_{n-1}\dots x\m_1 g})\m x_n)\notag\\
			&=(-1)^n\af_{\eta_1^g(x_1)}(1_{\eta_1^g(x_1)\m}w\circ\s_{n-1,n-1}^{\ol{x\m_1 g}}(x_2,\dots,x_n))\notag\\
			&\quad+1_{\s_{n,n}^g(x_1,\dots,x_n)}(\sum_{s=1}^{n-1}(-1)^{s+n} w\circ\s_{n-1,n-1}^g(x_1,\dots,x_sx_{s+1},\dots,x_n)\notag\\
			&\quad-w\circ\tau_n^g(x_1,\dots,x_n)).\label{w-Sg(n_n)}
		\end{align}
	\end{lem}
	\begin{proof}
		We first prove \cref{w-Sg(n_n)-n=1}:
		\begin{align*}
			0&=(\dl^1w) (\eta_1^g(x),(\ol{x\m g})\m)\\
			&=\alpha_{\eta_1^g(x)}(1_{\eta_1^g(x)\m}w((\ol{x\m g})\m))-1_{\eta_1^g(x)}w(g\m x)+1_{g\m x}w\circ\tau_1^g(x).
		\end{align*}
		
		To get \cref{w-Sg(n_n)}, write the following:
		\begin{align*}
			0&=(\dl^nw)\circ\s_{n,n}^g(x_1,\dots,x_n)\\  
			&=\af_{\eta_1^g(x_1)}(1_{\eta_1^g(x_1)\m}w\circ\s_{n-1,n-1}^{\ol{x\m_1 g}}(x_2,\dots,x_n))\\
			&\quad+\sum_{s=1}^{n-1}(-1)^s 1_{\eta_1^g(x_1\dots x_s)} w\circ\s_{n-1,n-1}^g(x_1,\dots,x_sx_{s+1},\dots,x_n)\\
			&\quad +(-1)^n 1_{\eta_1^g(x_1\dots x_n)} w(\tau_{n-1}^g(x_1,\dots,x_{n-1}),(\ol{x\m_{n-1}\dots x\m_1 g})\m x_n)\\
			&\quad +(-1)^{n+1} 1_{g\m x_1\dots x_n} w\circ\tau_n^g(x_1,\dots,x_n).
		\end{align*}
		It remains to multiply both sides by $1_{\s_{n,n}^g(x_1,\dots,x_n)}$. 
	\end{proof}
	
	\begin{lem}\label{w'-cohom-w-recursion}
		For all $n>0$, $w\in Z_{par}^n(G,\A)$ and $x_1,\dots,x_n\in G$:
		\begin{align}
			&(\dl^{n-1}\ve)(x_1,\dots,x_n)-\af_{x_1}(1_{x\m_1}\ve(x_2,\dots,x_n))-w(x_1,\dots,x_n)\notag\\
			&=1_{(x_1,\dots,x_n)}\prod_{g\in\Lb}\0_g\circ\af_{\eta_1^g(x_1)}\left(1_{\eta_1^g(x_1)\m}\sum_{j=0}^{n-1} (-1)^{j+1} w\circ\s_{n-1,j}^{\ol{x\m_1 g}}(x_2,\dots,x_n)\right)\notag\\
			&\quad -w'(x_1\dots,x_n).\label{delta-e-alpha-inv-w-inv=prod-w'-inv}
		\end{align}
	\end{lem}
	\begin{proof}
		Let $n=1$. Using \cref{delta-e-alpha-inv-w-inv-n=1,w-Sg(n_n)-n=1,0_g(a)=0_g(1_xa),w'-def,sigma_00-def,prod-0_g(1_(g-inv_x_1...x_n))}, we have
		\begin{align*}
			(\dl^0\ve)(x)-\af_x(1_{x\m}\ve)-w(x)&=1_x\prod_{g\in\Lb}\0_g(-w(g\m x))\\
			&=1_x\prod_{g\in\Lb}\0_g(-1_{\eta_1^g(x)}w(g\m x))\\
			&=-1_x\prod_{g\in\Lb}\0_g(\af_{\eta_1^g(x)}(1_{\eta_1^g(x)\m}w((\ol{x\m g})\m)))\\
			&\quad -1_x\prod_{g\in\Lb}\0_g(1_{g\m x}w\circ\tau_1^g(x))\\
			&=-1_x\prod_{g\in\Lb}\0_g\circ\af_{\eta_1^g(x)}(1_{\eta_1^g(x)\m}w(\s_{0,0}^{\ol{x\m g}}))\\
			&\quad -w'(x).
		\end{align*}
		
		For $n>1$ we use equalities \cref{-w+Sg(1_1),w-Sg(j_j),w-Sg(n_n)} multiplied by the idempotent
		\begin{align*}
			e=\prod_{i=1}^n 1_{\s_{n,i}^g(x_1,\dots,x_n)}=1_{(g\m,x_1,\dots,x_n)}\prod_{i=1}^n 1_{\eta_i^g(x_1,\dots,x_i)}.
		\end{align*}
		We have
		\begin{align}
			&e(-w(g\m x_1,x_2,\dots,x_n)+\Sg(1,1))\notag\\
			&=e\af_{\eta_1^g(x_1)}\left(1_{\eta_1^g(x_1)\m}\sum_{j=0}^{n-1} (-1)^{j+1} w\circ\s_{n-1,j}^{\ol{x\m_1 g}}(x_2,\dots,x_n)\right)\notag\\
			&\quad-e(w\circ\tau_n^g)(x_1,\dots,x_n)\notag\\
			&\quad+e\sum_{j=1}^{n-1}\sum_{i=j}^{n-1} (-1)^{i+j} w\circ\s_{n-1,i}^g(x_1,\dots,x_jx_{j+1},\dots,x_n)\label{prod_j=1^n-1-i=j^n-1}\\
			&\quad+e\sum_{j=2}^n\sum_{s=1}^{j-1} (-1)^{s+j} w\circ\s_{n-1,j-1}^g(x_1,\dots,x_sx_{s+1},\dots,x_n).\label{sum_j=1^n-1-s=1^j-1}
		\end{align}
		Introducing $j'=j-1$ in the sum \cref{sum_j=1^n-1-s=1^j-1}, we rewrite it as
		$$
		\sum_{j'=1}^{n-1}\sum_{s=1}^{j'} (-1)^{s+j'+1} w\circ\s_{n-1,j'}^g(x_1,\dots,x_sx_{s+1},\dots,x_n).
		$$
		Switching the order of summation, we get
		$$
		\sum_{s=1}^{n-1}\sum_{j'=s}^{n-1} (-1)^{s+j'+1} w\circ\s_{n-1,j'}^g(x_1,\dots,x_sx_{s+1},\dots,x_n),
		$$
		which is the opposite of the sum \cref{prod_j=1^n-1-i=j^n-1}. Hence,
		\begin{align}
			&e(-w(g\m x_1,x_2,\dots,x_n)+\Sg(1,1))\notag\\
			&=e\af_{\eta_1^g(x_1)}\left(1_{\eta_1^g(x_1)\m}\sum_{j=0}^{n-1} (-1)^{j+1} w\circ\s_{n-1,j}^{\ol{x\m_1 g}}(x_2,\dots,x_n)\right)\notag\\
			&\quad -e(w\circ\tau_n^g)(x_1,\dots,x_n).\label{e(-w(g-inv_x_1...x_n)+Sg(1_1))=e-af-e(w-tau)}
		\end{align}
		Since $\eta_i^g(x_1,\dots,x_i)\in H$ for all $i$, then after the application of $\0_g$ to the both sides of \cref{e(-w(g-inv_x_1...x_n)+Sg(1_1))=e-af-e(w-tau)}, we may remove $\prod_{i=1}^n 1_{\eta_i^g(x_1,\dots,x_i)}$ from $e$ by \cref{0_g(a)=0_g(1_xa)}. Moreover, using \cref{prod-0_g(1_(g-inv_x_1...x_n))}, we obtain
		\begin{align*}
			&1_{(x_1,\dots,x_n)}\prod_{g\in\Lb}\0_g(-w(g\m x_1,x_2,\dots,x_n)+\Sg(1,1))\\
			&=1_{(x_1,\dots,x_n)}\prod_{g\in\Lb}\0_g\circ\af_{\eta_1^g(x_1)}\left(1_{\eta_1^g(x_1)\m}\sum_{j=0}^{n-1} (-1)^{j+1} w\circ\s_{n-1,j}^{\ol{x\m_1 g}}(x_2,\dots,x_n)\right)\\
			&\quad -1_{(x_1,\dots,x_n)}\prod_{g\in\Lb}\0_g\circ w\circ\tau_n^g(x_1,\dots,x_n).
		\end{align*}
		
		Then \cref{delta-e-alpha-inv-w-inv=prod-w'-inv} follows by \cref{delta-e-alpha-inv-w-inv,w'-def}.
	\end{proof}
	
	The following lemma is~\cite[Lemma 3.10]{DKS}.
	\begin{lem}\label{apply-af-to-prod-and-switch-with-0}
		For all $x\in G$ and $a:\Lb'\to\A$ one has
		\begin{align}\label{af_x(1_x-inv-prod-0_g)=1_x-prod-0_g-circ-af}
			\af_x\left(1_{x\m}\prod_{g\in\Lb}\0_g(a(g))\right)=1_x\prod_{g\in\Lb}\0_g\circ\af_{\eta_1^g(x)}\left(1_{\eta_1^g(x)\m}a\left(\ol{x\m g}\right)\right).
		\end{align}
	\end{lem}
	
	\begin{lem}\label{alpha_eta_1^g(x_1)}
		For all $n>0$, $w\in Z_{par}^n(G,\A)$ and $x_1,\dots,x_n\in G$:
		\begin{align}
			&1_{(x_1\dots,x_n)}\prod_{g\in\Lb}\0_g\circ\alpha_{\eta_1^g(x_1)}\left(1_{\eta_1^g(x_1)\m}\sum_{j=0}^{n-1} (-1)^j w\circ\s_{n-1,j}^{\ol{x\m_1 g}}(x_2,\dots,x_n)\right)\notag\\
			&=\af_{x_1}\left(1_{x\m_1}\ve(x_2,\dots,x_n)\right).\label{prod_g-alpha_eta_1^g=alpha_x_1}
		\end{align}
	\end{lem}
	\begin{proof}
		Applying \cref{apply-af-to-prod-and-switch-with-0} with
		$$
		a(g)=\sum_{j=0}^{n-1} (-1)^j w\circ\s_{n-1,j}^g(x_2,\dots,x_n)
		$$
		(where $n$, $w$ and $x_2,\dots,x_n$ are fixed and $g\in\Lb'$), we see that the left-hand side of \cref{prod_g-alpha_eta_1^g=alpha_x_1} equals
		$$
		1_{(x_1\dots,x_n)}\af_{x_1}\left(1_{x\m_1}\prod_{g\in\Lb}\0_g\left(\sum_{j=0}^{n-1} (-1)^j w\circ\s_{n-1,j}^g(x_2,\dots,x_n)\right)\right).
		$$
		Since $1_{(x_1,\dots,x_n)}=\af_{x_1}\left(1_{x\m_1}1_{(x_2,\dots,x_n)}\right)$, we obtain \cref{prod_g-alpha_eta_1^g=alpha_x_1} by \cref{eps-def}.
	\end{proof}
	
	As a consequence of \cref{w'-cohom-w-recursion,alpha_eta_1^g(x_1)} we obtain the next.
	
	\begin{thrm}\label{w'-cohom-w}
		Let $n>0$ and $w\in Z_{par}^n(G,\A)$. Then $w=\dl^{n-1}\ve+w'$. In particular, $w'\in Z_{par}^n(G,\A)$.
	\end{thrm}	
	
	\subsection{Existence and uniqueness of a globalization}\label{sec:existsunique}
	
	Our aim in this section is to complete the construction of $\wtl w$ satisfying \cref{cocycle-ident-for-tilde-w,w-is-restr-of-tilde-w}. We start recalling the formulas from~\cite{DKS} which will be used here as well.
	
	\begin{lem}\label{formulas-for-eta_n^g}
		Let $g\in\Lb'$. Then
		\begin{align}
			\eta_n^g(x_1,\dots,x_n)&=\eta_{n-1}^{\ol{x\m_1g}}(x_2,\dots,x_n),\ n\ge 2,\label{eta_n^g=eta_n-1^ol(x-inv-g)}\\
			\eta_n^g(x_1,\dots,x_i,x_{i+1},\dots,x_n)&=\eta_{n-1}^g(x_1,\dots,x_ix_{i+1},\dots,x_n),\ 1\le i\le n-2,\label{eta_n^g-with-x_ix_i+1-glued}\\
			\eta_n^g(x_1,\dots,x_{n-1},x_nx_{n+1})&=\eta_n^g(x_1,\dots,x_n)\eta_{n+1}^g(x_1,\dots,x_{n+1}),\ n\ge 1.\label{eta_n^g-eta_n+1^g}
		\end{align}
	\end{lem}
	
	Define $\wtl{w'}:G^n\to\A$ the same way as it was done in~\cite{DKS}, i.e. by removing $1_{(x_1,\dots,x_n)}$ from \cref{w'-def}:
	\begin{align}\label{w'-tilde-def}
		\wtl{w'}(x_1,\dots,x_n)=\prod_{g\in\Lb}\0_g\circ w\circ\tau_n^g(x_1,\dots,x_n).
	\end{align}
	
	\begin{lem}\label{w'-tilde-is-quasi-cocycle}
		Let $n>0$, $w\in Z_{par}^n(G,A)$ and $x_1,\dots,x_n\in G$. Then
		\begin{align}\label{cocycle-ident-for-tilde-w'}
			\tl\dl^n\wtl{w'}=0.
		\end{align}
	\end{lem}
	\begin{proof}
		Using \cref{w'-tilde-def,af_x(1_x-inv-prod-0_g)=1_x-prod-0_g-circ-af,(tl-dl^n-f)(g_1...g_n)}, we rewrite the left-hand side of \cref{cocycle-ident-for-tilde-w'} as follows
		\begin{align}
			&1_{x_1}\prod_{g\in\Lb}\0_g\circ\af_{\eta_1^g(x_1)}\left(1_{\eta_1^g(x_1)\m}w\circ\tau_n^{\ol{x\m_1 g}}(x_2,\dots,x_{n+1})\right)\notag\\
			&\quad+1_{x_1}\prod_{g\in\Lb}\0_g\left(\sum_{i=1}^n(-1)^i w\circ\tau_n^g(x_1,\dots,x_ix_{i+1},\dots,x_{n+1})\right)\label{prod_i=1^n-prod_g-in-Lb}\\
			&\quad+1_{x_1}\prod_{g\in\Lb}\0_g\left((-1)^{n+1} w\circ\tau_n^g(x_1,\dots,x_n)\right).\label{prod_g-in-Lb-0_g-circ-w-circ-tau_n^g}
		\end{align}
		Notice by \cref{0_g(a)=0_g(1_xa)} that one can multiply the $i$-th term of the sum in \cref{prod_i=1^n-prod_g-in-Lb} by $1_{\eta_1^g(x_1\dots x_i)}$.  And similarly, the argument of $\0_g$ in \cref{prod_g-in-Lb-0_g-circ-w-circ-tau_n^g} can be multiplied by $1_{\eta_1^g(x_1\dots x_n)}$ without changing the value of the map. Hence, for \cref{cocycle-ident-for-tilde-w'}, it suffices to prove
		\begin{align}
			0&=\af_{\eta_1^g(x_1)}\left(1_{\eta_1^g(x_1)\m}w\circ\tau_n^{\ol{x\m_1 g}}(x_2,\dots,x_{n+1})\right)\notag\\
			&\quad+\sum_{i=1}^n (-1)^i 1_{\eta_1^g(x_1\dots x_i)} w\circ\tau_n^g(x_1,\dots,x_ix_{i+1},\dots,x_{n+1})\notag\\
			&\quad+(-1)^{n+1}1_{\eta_1^g(x_1\dots x_n)}w\circ\tau_n^g(x_1,\dots,x_n).\label{delta^nw-circ-ta^g_n=0-expanded}
		\end{align}
		As in the proof of~\cite[Lemma~4.2]{DKS}, one shows using \cref{eta_n^g=eta_n-1^ol(x-inv-g),eta_n^g-with-x_ix_i+1-glued,eta_n^g-eta_n+1^g} that \cref{delta^nw-circ-ta^g_n=0-expanded} is an expansion of the cocycle identity
		\begin{align*}
			(\dl^nw)\circ\tau^g_{n+1}(x_1,\dots,x_{n+1})=0.
		\end{align*}
	\end{proof}
	
	Now, for all $n>0$ and $x_1,\dots,x_n\in G$, define $\wtl w:G^n\to\A$ as the sum
	\begin{align}\label{tilde-w-in-terms-of-tl-w'} 
		\wtl w=\wtl{w'}+\hat\dl^{n-1}\ve,
	\end{align} 
	where
	\begin{align}\label{(hat-dl^n-ve)(x_1...x_n)} 
		(\hat\dl^{n-1}\ve)(x_1,\dots,x_n)&=\af_{x_1}(1_{x\m_1}\ve(x_2,\dots,x_n))\notag\\  
		&\quad+\sum_{i=1}^{n-1}(-1)^i \ve(x_1,\dots,x_i x_{i+1},\dots,x_n)\notag\\ 
		&\quad+(-1)^n \ve(x_1,\dots,x_{n-1}).
	\end{align}  
	
	The existence of a globalization is established in the following theorem.
	\begin{thrm}\label{glob-exists} 
		Let $\A$ be a direct product of indecomposable unital rings and $\af=\{\af_g:\cD_{g\m}\to\cD_g\mid g\in G\}$ a (non-necessarily transitive) unital partial action of $G$ on $\A$. Then for any $n\ge 0$ each cocycle $w\in Z_{par}^n(G,\A)$ with values in the induced $\kpar G$-module is globalizable. 
	\end{thrm}
	\begin{proof} 
		The case $n=0$ is the existence part of \cref{0-cocycle-is-globalizable}, so let $n>0$. As in the proof of~\cite[Theorem~6.3]{DKS}, it is enough to consider transitive $\af$. The map $\wtl w:G^n\to\A$ defined in \cref{tilde-w-in-terms-of-tl-w'} satisfies \cref{w-is-restr-of-tilde-w}, as 
		\begin{align*}
			1_{(x_1,\dots,x_n)}\wtl w(x_1,\dots,x_n)&=1_{(x_1,\dots,x_n)}\wtl{w'}(x_1,\dots,x_n)+1_{(x_1,\dots,x_n)}(\hat\dl^{n-1}\ve)(x_1,\dots,x_n)\\
			&=w'(x_1,\dots,x_n)+(\dl^{n-1}\ve)(x_1,\dots,x_n)\\ 
			&=w(x_1,\dots,x_n).
		\end{align*}
		for all $x_1,\dots,x_n\in G$ by \cref{w'-def,w'-tilde-def,(tl-dl^n-f)(g_1...g_n),w'-cohom-w,(dl^nf)(g_1...g_(n+1))}.
		
		To apply \cref{w-glob-iff-exists-tilde-w}, it remains to prove \cref{cocycle-ident-for-tilde-w}. Observe by \cref{tilde-w-in-terms-of-tl-w'} that 
		\begin{align*}
			(\tl\dl^n\wtl w)(x_1,\dots,x_{n+1})=(\tl\dl^n\wtl{w'})(x_1,\dots,x_{n+1})+(\tl\dl^n\hat\dl^{n-1}\ve)(x_1,\dots,x_{n+1}),
		\end{align*}
		where the first summand is zero by \cref{w'-tilde-is-quasi-cocycle}. Comparing $\tl\dl^n\hat\dl^{n-1}\ve$ with the classical $\dl^n\dl^{n-1}\ve$, we see that the difference is that instead of a global action one has $\af$, and all the terms to which $\af$ is not applied are multiplied by $1_{x_1}$. The latter terms cancel, and the remaining ones are
		\begin{equation}\label{af_x_1(1_x_1-inv(tl-delta^n-1e))}
			\af_{x_1}\left(1_{x\m_1}(\hat\dl^{n-1}\ve)(x_2,\dots,x_{n+1})\right)
		\end{equation} 
		and the first summands in each
		\begin{align}
			&(-1)^i1_{x_1}(\hat\dl^{n-1}\ve)(x_1,\dots,x_ix_{i+1},\dots,x_{n+1}),\ 1\le i\le n,\label{(tl-delta^n-1-e)(x_1_..._x_ix_i+1_..._x_n+1)}\\
			&(-1)^{n+1}1_{x_1}(\hat\dl^{n-1}\ve)(x_1,\dots,x_n).\label{(tl-delta^n-1-e)(x_1_..._x_n)}
		\end{align} 
		The terms of the expansion of \cref{af_x_1(1_x_1-inv(tl-delta^n-1e))} are
		\begin{align}
			&\af_{x_1}\left(1_{x\m_1}\af_{x_2}\left(1_{x\m_2}\ve(x_3,\dots,x_{n+1})\right)\right),\label{af_x_1(1_x_1-inv-tl-af_x_2(tl-e(x_3_..._x_n+1)))}\\
			&\af_{x_1}\left((-1)^{i-1} 1_{x\m_1}\ve(x_2,\dots,x_i x_{i+1},\dots,x_{n+1})\right),\ 2\le i\le n,\label{af_x_1(1_x_1-inv-tl-e(x_2_..._x_i x_i+1_..._x_n+1))}\\
			&(-1)^n\af_{x_1}\left(1_{x\m_1}\ve(x_2,\dots,x_n)\right),\label{af_x_1(1_x_1-inv-tl-e(x_2_..._x_n))}
		\end{align} 
		while the first summands in \cref{(tl-delta^n-1-e)(x_1_..._x_ix_i+1_..._x_n+1),(tl-delta^n-1-e)(x_1_..._x_n)} are
		\begin{align}
			&-1_{x_1}\af_{x_1x_2}\left(1_{x\m_2 x\m_1}\ve(x_3,\dots,x_{n+1})\right),\label{tl-af_x_1x_2(tl-e(x_3_..._x_n+1))}\\
			&(-1)^i\af_{x_1}\left(1_{x\m_1}\ve(x_2,\dots,x_ix_{i+1},\dots,x_{n+1})\right),\ 2\le i\le n,\label{tl-af_x_1(tl-e(x_2_..._x_ix_i+1_..._x_n+1))}\\
			&(-1)^{n+1}\af_{x_1}\left(1_{x\m_1}\ve(x_2,\dots,x_n)\right).\label{tl-af_x_1(tl-e(x_2_..._x_n))}
		\end{align} 
		Thus, \cref{af_x_1(1_x_1-inv-tl-af_x_2(tl-e(x_3_..._x_n+1)))} cancels with \cref{tl-af_x_1x_2(tl-e(x_3_..._x_n+1))}, \cref{af_x_1(1_x_1-inv-tl-e(x_2_..._x_i x_i+1_..._x_n+1))} with \cref{tl-af_x_1(tl-e(x_2_..._x_ix_i+1_..._x_n+1))}, and \cref{af_x_1(1_x_1-inv-tl-e(x_2_..._x_n))} with \cref{tl-af_x_1(tl-e(x_2_..._x_n))}.
	\end{proof}
	
	We recall the following result from~\cite{DKS}.
	\begin{prop}\label{B-embeds-into-prod_(g-in-Lb')A_g}
		Let $\A$ be a direct product $\prod_{g\in\Lb}\A_g$ of indecomposable unital rings, $\af$ a transitive unital partial action of $	G$ on $\A$ and $(\bt,\B)$ an enveloping action of $(\af,\A)$ with $\A\sst\B$. Then $\B$ embeds as an ideal into $\prod_{g\in\Lb'}\A_g$, where $\A_g$ denotes the ideal $\bt_g(\A_1)$ in $\B$. Moreover, $\M(\B)\cong\prod_{g\in\Lb'}\A_g$, and $\bt^*$ is transitive, when seen as a partial action of $G$ on $\prod_{g\in\Lb'}\A_g$.
	\end{prop}
	
	This permits us to obtain the uniqueness of a globalization.
	\begin{thrm}\label{uniqueness-of-glob}
		Let $\A$ be a direct product $\prod_{g\in\Lb}\A_g$ of indecomposable unital rings, $\af$ a unital partial action of $G$ on $\A$ and $w_i\in Z_{par}^n(G,\A)$, $i=1,2$ ($n>0$). Suppose that $(\bt,\B)$ is an enveloping action of $(\af,\A)$ and $u_i\in Z^n(G,\M(\B))$ is a globalization of $w_i$, $i=1,2$. If $w_1$ is cohomologous to $w_2$, then $u_1$ is cohomologous to $u_2$. In particular, any two globalizations of the same partial $n$-cocycle are cohomologous.
	\end{thrm}
	\begin{proof}
		As in the proof of~\cite[Theorem~5.3]{DKS}, we consider only the transitive case and assume, without loss of generality, that $\M(\B)=\prod_{g\in\Lb'}\A_g\supseteq\A$. Then we define
		the homomorphism $\vt_g:\M(\B)\to\M(\B)$ by
		\begin{align}\label{vt_g-def}
			\vt_g=\bt_g\circ\pr_1
		\end{align}
		and $u'_i\in C^n(G,\M(\B))$ by
		\begin{align}\label{u'_i-def}
			u'_i(x_1,\dots,x_n)=\prod_{g\in\Lb'}\vt_g\circ u_i\circ\tau_n^g(x_1,\dots,x_n),\ i=1,2.
		\end{align}
		We infer that $u'_i\in Z^n(G,\M(\B))$ and $u_i$ is cohomologous to $u'_i$, $i=1,2$, because the definition of $u'_i$ is totally analogous to that of $w'$ (compare \cref{u'_i-def} with \cref{w'-def} and \cref{vt_g-def} with \cref{0_g-def}, see also \cref{w'-cohom-w}). 
		
		Suppose that $w_1$ is cohomologous to $w_2$. To prove that $u_1$ is cohomologous to $u_2$, it is enough to establish the cohomological equivalence of $u'_1$ and $u'_2$. Observe as in the proof of~\cite[Theorem~5.3]{DKS} that
		\begin{align}\label{u'_i-in-terms-of-w} 
			u'_i(x_1,\dots,x_n)=\prod_{g\in\Lb'}\vt_g\circ w_i\circ\tau_n^g(x_1,\dots,x_n),\ i=1,2.
		\end{align}
		If $w_2=w_1+\dl^{n-1}\xi$ for some $\xi\in C_{par}^{n-1}(G,\A)$, then by \cref{u'_i-in-terms-of-w} one readily gets $u'_2=u'_1+(\dl^{n-1}\xi)'$, where
		$$
		(\dl^{n-1}\xi)'(x_1,\dots,x_n)=\prod_{g\in\Lb'}\vt_g\circ (\dl^{n-1}\xi)\circ\tau_n^g(x_1,\dots,x_n).
		$$
		It is enough to prove that
		\begin{align}\label{(dl^(n-1)-xi)'=dl^(n-1)-xi'}
			(\dl^{n-1}\xi)'=\dl^{n-1}\xi',
		\end{align}
		where
		$$
		\xi'(x_1,\dots,x_{n-1})=\prod_{g\in\Lb'}\vt_g\circ\xi\circ\tau_{n-1}^g(x_1,\dots,x_{n-1}).
		$$
		Expanding $(\dl^{n-1}\xi)'(x_1,\dots,x_n)$, we see that the idempotents $1_{\eta_i^g(x_1\dots x_i)}$ which appear in the expansion may be removed, because the analogue of \cref{0_g(a)=0_g(1_xa)} holds for $\vt_g$ too (compare \cref{vt_g-def} with \cref{0_g-def}). Thus, due to the fact that $\vt_g$ is a homomorphism, \cref{(dl^(n-1)-xi)'=dl^(n-1)-xi'} reduces to
		\begin{align*}
			&\beta_{x_1}\left(\prod_{g\in\Lb'}\vt_g\circ\xi\circ\tau^g_{n-1}(x_2,\dots,x_n)\right)\notag\\
			&=\prod_{g\in\Lb'}\vt_g\circ \bt_{\eta^g(x_1)}\circ\xi(\eta^g_2(x_1,x_2),\dots,\eta^g_n(x_1,\dots,x_n))\\
			&=\prod_{g\in\Lb'}\vt_g\circ \bt_{\eta^g(x_1)}\circ\xi\circ\tau^{\ol{x\m_1g}}_{n-1}(x_2,\dots,x_n),
		\end{align*}
		which follows from the global version of \cref{apply-af-to-prod-and-switch-with-0}.
	\end{proof}
	
	\begin{cor}\label{H^n(G_A)-cong-H^n(G_B)}
		Let $\A$ be a direct product $\prod_{g\in\Lb}\A_g$ of indecomposable unital rings, $\af$ a partial action of $G$ on $\A$ and $(\bt,\B)$ an enveloping action of $(\af,\A)$. Then $H_{par}^n(G,\A)$ is isomorphic to the classical cohomology group $H^n(G,\M(\B))$.
	\end{cor}
	\noindent Indeed, we know by \cref{restr-hom-of-cohom-groups} that for all $n\ge 0$ there is a homomorphism $H^n(G,\M(\B))\to H^n_{par}(G,\A)$ coming from the restriction map. By \cref{glob-exists,uniqueness-of-glob} this homomorphism is invertible, when $n>0$. The case $n=0$ is \cref{0-cocycle-is-globalizable}.
	
	\section*{Acknowledgments}
	This work was partially supported by CNPq of Brazil (Proc. 307873/2017-0, 404649/2018-1),  FAPESP of Brazil (Proc. 2015/09162-9), MINECO (MTM2016-77445-P), Fundaci\'on S\'eneca of Spain and Funda\c{c}\~ao para a Ci\^{e}ncia e a Tecnologia (Portuguese Foundation for Science and Technology) through the project PTDC/MAT-PUR/31174/2017. The first two authors would also like to thank the Department of Mathematics of the University of Murcia for its warm hospitality during their visits.
	
	\bibliography{bibl-pact}{}
	\bibliographystyle{acm}
	
\end{document}